\documentclass{amsart}

\usepackage{enumerate}
\usepackage{amsmath, amscd}
\usepackage{ifpdf}
\usepackage{amsfonts}
\usepackage[normalem]{ulem}
\usepackage{color}
\usepackage{amssymb}
\usepackage{amsthm}
\usepackage[hyperfootnotes=false]{hyperref}
\usepackage{setspace}
\usepackage{graphicx}

\newtheorem{thm}{Theorem}
\newtheorem{main theorem}[thm]{Main Theorem}
\newtheorem{corollary}[thm]{Corollary}

\newtheorem{lemma}[thm]{Lemma}
\newtheorem{prop}[thm]{Proposition}
\newtheorem{conjecture}[thm]{Conjecture}

\theoremstyle{definition}

\newtheorem{defn}[thm]{Definition}
\newtheorem{remark}[thm]{Remark}

\newcommand{\bea}{\begin{eqnarray*}}
\newcommand{\eea}{\end{eqnarray*}}

\newcommand{\be}{\begin{equation}}
\newcommand{\ee}{\end{equation}}

\ifpdf
\title[Adaptive trains]{Adaptive trains for attracting sequences of holomorphic automorphisms}

\author[Peters, Smit]{Han Peters, Iris Marjan Smit}

\address{KdV Institute for Mathematics\\
University of Amsterdam\\
The Netherlands}
\email{h.peters@uva.nl}
\email{imsmit@gmail.com}

\begin{document}
\bibliographystyle{plain}

\begin{abstract}
Consider a holomorphic automorphism acting hyperbolically on an invariant compact set. It has been conjectured that the arising stable manifolds are all biholomorphic to Euclidean space. Such a stable manifold is always equivalent to the basin of a uniformly attracting sequence of maps. The equivalence of such basins to Euclideans has been shown under various additional assumptions. Recently Majer and Abbondandolo achieved new results by non-autonomously conjugating to normal forms on larger and larger time intervals. We show here that their results can be improved by adapting these time intervals to the sequence of maps. Under the additional assumption that all maps have linear diagonal part the adaptation is quite natural and quickly leads to significant improvements. We show how this construction can be emulated in the non-diagonal setting.
\end{abstract}

\maketitle

\tableofcontents

\section{Introduction}

Let $X$ be a complex manifold equipped with a Riemannian metric, and let $f:X \rightarrow X$ be an automorphism that acts hyperbolically on some invariant compact subset $K \subset X$. Let $p \in K$ and write $\Sigma_f^s(p)$ for the stable manifold of $f$ through $p$. $\Sigma_f^s(p)$ is a complex manifold, say of complex dimension $m$. In the special case where $p$ is a fixed point it is known that $\Sigma_f^s(p)$ is biholomorphically equivalent to $\mathbb C^m$. It was conjectured by Bedford \cite{Bedford} that this equivalence holds for any $p \in K$.

\begin{conjecture}[Bedford] \label{conj:stable}
The stable manifold $\Sigma_f^s(p)$ is always equivalent to $\mathbb C^m$.
\end{conjecture}

The usual approach towards this problem is to translate it to a non-autonomous setting. Let $f_0, f_1, \ldots$ be a sequence of automorphisms of $\mathbb{C}^m$ satisfying
\begin{equation}\label{eq:uniform}
C\|z\| \le \|f_n(z)\| \le D\|z\|,
\end{equation}
for all $n \in \mathbb{N}$ and all $z$ lying in the unit ball $\mathbb{B}$, where the constants $1>D>C>0$ are independent of $n$. We will call a sequence $(f_n)$ satisfying condition \eqref{eq:uniform} \emph{uniformly attracting}. We define the basin of attraction of the sequence $(f_n)$ as
\begin{equation}
\Omega = \Omega_{(f_n)} = \{z \in \mathbb{C}^k \mid f_n \circ \cdots \circ f_0(z) \rightarrow 0\}.
\end{equation}

It was shown by Forn{\ae}ss and Stens{\o}nes that a positive answer to the following conjecture implies a positive answer to Conjecture \ref{conj:stable}.

\begin{conjecture} \label{conj:main}
The basin $\Omega$ is always biholomorphic to $\mathbb{C}^m$.
\end{conjecture}

Here we present the following new results, both giving positive answers to Conjecture \ref{conj:main} under additional hypotheses.

\begin{thm}\label{thm:main}
Let $(f_n)$ be a uniformly attracting sequence of automorphisms of $\mathbb C^2$, and suppose that the maps $f_n$ all have order of contact $k$. Assume further that the linear part of each map $f_n$ is diagonal and that
$$
D^{k+1} < C.
$$
Then the basin of attraction $\Omega_{(f_n)}$ is biholomorphic to $\mathbb C^2$.
\end{thm}

\begin{thm}\label{thm:general}
Let $(f_n)$ be a uniformly attracting sequences of automorphisms of $\mathbb C^2$, and suppose that
$$
D^{11/5} < C.
$$
Then the basin of attraction $\Omega_{(f_n)}$ is biholomorphic to $\mathbb C^2$.
\end{thm}

In the next section we will place these two results in a historical perspective. While Theorems \ref{thm:main} and \ref{thm:general} are significant improvements over previously known results, our main contribution lies in a new way of thinking about \emph{trains}, as introduced by Abbondandolo and Majer in \cite{AM}.

\medskip

The authors would like to thank Alberto Abbondandolo, Leandro Arosio, John Erik Forn{\ae}ss, Jasmin Raissy, Pietro Majer and Liz Vivas for many stimulating discussions. The first author was supported by a SP3-People Marie Curie Actionsgrant in the project Complex Dynamics (FP7-PEOPLE-2009-RG, 248443).

\section{A short history}

Let $f$ be an automorphism of a complex manifold $X$ of dimension $m$, and let $p \in X$ be an attracting fixed point. Define the basin of attraction by
$$
\Omega = \{ z \in X \; \mid \; f^n(z) \rightarrow p\}.
$$
The following result was proved independently by Sternberg \cite{Sternberg} and Rosay-Rudin \cite{RR}.

\begin{thm}[Sternberg, Rosay-Rudin]
The basin $\Omega$ is biholomorphic to $\mathbb C^m$.
\end{thm}

If $p$ is not attracting but hyperbolic, one considers the restriction of $f$ to the stable manifold. This restriction is an automorphism of $\Sigma^s_f(p)$ with an attracting fixed point at $p$. Moreover, its basin of attraction is equal to the entire stable manifold, which is therefore equivalent to $\mathbb C^m$. This naturally raised Conjecture \ref{conj:stable}. Equivalence to $\mathbb C^m$ of \emph{generic} stable manifolds was proved by Jonsson and Varolin in \cite{JV}.

\begin{thm}[Jonsson-Varolin]\label{thm:JV}
Let $X$, $f$ and $K$ be as in Conjecture \ref{conj:stable}. For a generic point $p \in K$ the stable manifold through $p$ is equivalent to $\mathbb C^m$.
\end{thm}

Here generic refers to a subset of $K$ which has full measure with respect to any invariant probability measure on $K$. In fact, Jonsson and Varolin showed that Conjecture \ref{conj:stable} holds for so-called Oseledec points. The results of Jonsson and Varolin were extended by Berteloot, Dupont and Molino in \cite{BDM}. More recently the following was shown in \cite{AAM}:

\begin{thm}[Abate-Abbondandolo-Majer]\label{thm:Lyapunov}
The existence of Lyapunov exponents is enough to guarantee $\Sigma_f^s(p) \cong \mathbb C^m$.
\end{thm}

Shortly after the positive result of Jonsson and Varolin, Forn{\ae}ss proved the following negative result, which led many people to believe that Conjectures \ref{conj:stable} and \ref{conj:main} must be false. Consider a sequence of maps $(f_n)_{n\ge 0}$ given by
$$
f_n: (z,w) \rightarrow (z^2 + a_n w, a_n z),
$$
where $|a_0| < 1$ and $|a_{n+1}| \le |a_n|^2$.

\begin{thm}[Forn{\ae}ss]\label{thm:short}
The basin $\Omega_{(f_n)}$ is not biholomorphic to $\mathbb C^2$. Indeed there exists a bounded plurisubharmonic function on $\Omega_{(f_n)}$ which is not constant.
\end{thm}

We note that this result does not give a counterexample to Conjecture \ref{conj:main}, as the sequence $(f_n)$ violates the condition
$$
C\|z\| \le \|f_n(z)\| \le D \|z\|
$$
on any uniform neighborhood of the origin. In Theorem \ref{thm:short} the rate of contraction is not uniformly bounded from below.

If the bounds $C$ and $D$ satisfy $D^2 < C$ then it was shown by Wold in \cite{Wold} that the basin of attraction is biholomorphic to $\mathbb C^m$. This result was generalized by Sabiini \cite{Sabiini} to the following statement.

\begin{thm}\label{thm:wold}
Let $(f_n)$ be a sequence of automorphisms which satisfies the conditions in Conjecture \ref{conj:main}, and suppose that $D^k < C$ for some $k \in \mathbb N$. Suppose further that the maps $f_n$ all have order of contact $k$. Then the basin of attraction $\Omega_{(f_n)}$ is biholomorphic to $\mathbb C^m$.
\end{thm}

Here a map $f$ has order of contact $k$ if $f(z) = l(z) + O(\|z\|^k)$, where $l(z)$ is linear.

The following was recently shown in \cite{AM}:

\begin{thm}[Abbondandolo-Majer]\label{thm:AM}
There exists a constant $\epsilon  = \epsilon(m) > 0$ such that
the following holds: For any sequence $(f_n)$ which satisfies the
conditions in Conjecture \ref{conj:main} with
$$
D^{2+\epsilon} < C,
$$
the basin $\Omega_{(f_n)}$ is biholomorphic to $\mathbb C^m$.
\end{thm}

It was noted in \cite{AM} that for maps with order of contact $k$, one
should be able to obtain the same result if $D^{k+\epsilon}<C$, where
$\epsilon=\epsilon(m,k)$.

While this result only seems marginally stronger (indeed, $\epsilon$-stronger), it is in fact a much deeper result. If $D^k < C$ one can ignore all but the linear terms, see Lemma \ref{lemma:conjugation} in the next section. But if $D^k \ge C$, one has to deal with the terms of order $k$, which is a major difficulty. In fact, looking at more classical results in local complex dynamics, the major difficulty in describing the behavior near a fixed point usually lies in controlling the lowest order terms which are not trivial. Theorems \ref{thm:AM} and \ref{thm:main} may therefore be important steps towards a complete understanding of Conjecture \ref{conj:main}.

Now that techniques have been found to deal with these terms of degree $k$, it is natural to ask whether the condition $D^{k+\epsilon} < C$ can be pushed to the condition $D^{k+1} < C$, since as long as $D^{k+1} < C$ is satisfied one can ignore terms of degree strictly greater than $k$. In Theorem \ref{thm:main} we show that we can indeed weaken the requirement to $D^{k+1} < C$, at least in $2$ complex dimensions and under the additional assumption that the linear parts of all the maps $f_n$ are diagonal.

Whether diagonality is a serious extra assumption or merely simplifies the computations remains to be seen, but we will see that some of the techniques we introduce to prove Theorem \ref{thm:main} can be applied without the diagonality assumption as well, leading to Theorem \ref{thm:general}.

Since the computations in the general case are much more intensive than in the diagonal case, we have chosen not to prove Theorem \ref{thm:general} for maps with higher order of contact. We expect that the interested reader will be able to generalize Theorem \ref{thm:general} to maps with higher order of contact.

We conclude our historical overview with the following two results. The first was proved in \cite{PW}.

\begin{thm}[Peters-Wold]\label{thm:repeat}
Let $(f_j)$ be a sequence of automorphism of $\mathbb C^m$, each with an attracting fixed point at the origin. Then there exists a sequence of integers $(n_j)$ such that the attracting basin of the sequence $(f_j^{n_j})$ is equivalent to $\mathbb C^m$.
\end{thm}

The next result, from \cite{Peters}, will be important to us not so much for the statement itself but for the ideas used in the proof. See also the more technical Lemma \ref{lemma:dominant}, which will be discussed in the next section.

\begin{thm}[Peters]\label{thm:perturbations}
Let $F$ be an automorphism of $\mathbb C^m$. Then there exists an $\epsilon > 0$ such that for any sequence $(f_n)$ of automorphisms of $\mathbb C^m$ which all fix the origin and satisfy $\|f_n - F\|_{\mathbb B} < \epsilon$ for all $n$, one has that $\Omega_{(f_n)} \cong \mathbb C^m$.
\end{thm}

The combination of Theorems \ref{thm:repeat} and \ref{thm:perturbations} provides a good idea for the techniques used to prove Theorems \ref{thm:main} and \ref{thm:general}: If, after applying suitable changes of coordinates, we can find sufficiently long stretches in the sequence $(f_n)$ that behave similarly, then we can show that the basin will be biholomorphic to $\mathbb C^m$. What is meant by \emph{behave similarly} will be made more precise in later sections.

\subsection{Applications}

The theorems described above have been used to prove a number of
results which are at first sight unrelated. The first example of
such an application is of course the classical result of Fatou and
Bieberbach which states that there exists a proper subdomain of
$\mathbb C^2$ which is biholomorphic to $\mathbb C^2$. Indeed, it
is not hard to find an automorphism of $\mathbb C^2$ with an
attracting fixed point, but whose basin of attraction is not equal
to the entire $\mathbb C^2$. Proper subdomains of $\mathbb C^2$
that are equivalent to $\mathbb C^2$ are now called
\emph{Fatou-Bieberbach} domains.

It should be no surprise that for the construction of Fatou-Bieberbach domains with specific properties, it is useful to work with non-autonomous basins. Working with sequences of maps gives much more freedom than working with a single automorphism. Using Theorem \ref{thm:wold} it is fairly easy to construct a sequence of automorphisms satisfying various global properties, while making sure that the attracting basin is equivalent to $\mathbb C^2$. We give two examples of results that have been obtained in this way.

\begin{thm}[Wold]
There exist countably many disjoint Fatou-Bieberbach domains in $\mathbb C^2$ whose union is dense.
\end{thm}

The question of whether such Fatou-Bieberbach domains exist was raised by Rosay and Rudin in \cite{RR}.

\begin{thm}[Peters-Wold]
There exists a Fatou-Bieberbach domain $\Omega$ in $\mathbb C^2$ whose boundary has Hausdorff dimension $4$, and positive Lebesgue measure.
\end{thm}

Here we note that Fatou-Bieberbach domains with Hausdorff dimension equal to to any $h \in (3,4)$ were constructed by Wolf \cite{Wolf} using autonomous attracting basins. Hausdorff dimension $3$ (and in fact $C^\infty$-boundary) was obtained earlier by Stens{\o}nes in \cite{Stensones}, who also used an iterative procedure involving a sequence of automorphisms of $\mathbb C^2$.

The last application we would like to mention is the Loewner partial differential equation. The link with non-autonomous attracting basins was made in \cite{Arosio}, where Arosio used a construction due to Forn{\ae}ss and Stens{\o}nes (see Theorem \ref{thm:FS} below) to prove the existence of solutions to the Loewner PDE. See also \cite{ABW} for the relationship between the Loewner PDE and non-autonomous attracting basins.

\section{Techniques}

\subsection{The autonomous case}

Essentially the only available method for proving that a domain $\Omega$ is equivalent to $\mathbb C^m$, is by explicitly constructing a biholomorphic map from $\Omega$ to $\mathbb C^m$. Let us first review how this is done for autonomous basins, before we look at how this proof can be adapted to the non-autonomous setting. We follow the proof in the appendix of \cite{RR}, but see also the survey \cite{Berteloot} written by Berteloot.

Given an automorphism $F$ of $\mathbb C^m$ with an attracting fixed point at the origin, one can find  a \emph{lower triangular polynomial map} $G$, and for any $k \in \mathbb N$, a polynomial map $X_k$ of the form $X_k = \mathrm{Id} + h.o.t.$, biholomorphic in a neighborhood of the origin, such that
\begin{equation}\label{eq:conjugate}
X_k \circ F = G \circ X_k + O(\|z\|^{k+1}).
\end{equation}
Here a polynomial map $G = (G_1, \ldots , G_m)$ is called \emph{lower triangular} if
$$
G_i = \lambda_i z_i + H_i(z_1, \ldots z_{i-1})
$$
for $i= 1, \ldots , m$. Now consider the sequence of holomorphic maps  from $\Omega_F$ to $\mathbb C^m$ given by
$$
\Phi_n = G^{-n} \circ X_k\circ  F^n.
$$
Important here is that the lower triangular polynomial maps behave very similarly to linear maps. For example, it follows immediately by induction on $m$ that the degrees of the iterates $G^n$ are uniformly bounded. One can also easily see that the basin of attraction of a lower triangular map with an attracting fixed point at the origin is always equal to the whole set $\mathbb C^m$.

\begin{lemma}
Let $G$ be a lower triangular. Then there exist a $\beta>0$ such that
$$
\|G^{-1}(z) - G^{-1}(w)\| \le \beta \|z-w\|
$$
for all $z, w \in \mathbb B$.
\end{lemma}

Now let $D>0$ be such that $\|F(z)\|< D\|z\|$ for $z \in \mathbb B$. Then there exists a $k \in \mathbb N$ such that $D^{k+1}\cdot \beta < 0$. It follows from Equation \eqref{eq:conjugate} that, with this choice of $k$, the maps $\Phi_n$ converge, uniformly on compact subsets of $\Omega_F$, to a biholomorphic map from $\Omega_F$ to $\mathbb C^m$.

While the proof outlined above works elegantly, in the non-autonomous case it will often not be easy to control the maps $X_k$ for arbitrarily high $k$. However, it turns out that with a little more care it is sufficient to work with a much lower $k$, that really only depends on the eigenvalues of $DF(0)$. Let $0<C<D<1$ be such that
$$
C\|z\| < \|F(z)\| \le D\|z\|,
$$
and let $k$ be such that $D^{k+1} < C$. It turns out that it is
sufficient to work with the maps $X_k$. To see this, let
$C^\prime<C$ be such that $D^{k+1} < C^\prime$. Then there exist
an $r>0$ sufficiently small such that for $z, w \in
\mathbb{B}(0,r)$ one has
$$
\|G^{-1}(z) - G^{-1}(w)\| \le (C^\prime)^{-1}\|z-w\|.
$$
Let $K$ be a relatively compact subset of the attracting basin.
Then there exists an $N$ such that $F^N(K) \subset
\mathbb{B}(0,\frac{r}{2})$. If $r$ is chosen sufficiently small
then we easily see that for any $m \ge N$ and $j \le m-N$ we have
that
$$
G^{-j}\circ X_k \circ F^m(K) \subset \mathbb{B}(0,r).
$$
It follows that
$$
\|\Phi_{n+1} - \Phi_n\|_K \le \beta^N {C^{\prime}}^{-n+N}D^{(n-N)(k+1)},
$$
which is summable over $n$. Hence the sequence $(\Phi_n)$ forms a
Cauchy sequence and converges to a limit map $\Phi$. Since
$D\Phi_n(0) = \mathrm{Id}$ for all $n$, it follows that $\Phi$ is
biholomorphic onto its image. To prove the surjectivity of $\Phi$
one shows that for each $\mathbb{B}(R) \subset \mathbb C^m$ there
exists a compact $K \subset \Omega_F$ such that $\Phi_n(K) \supset
\mathbb{B}(R)$ for all $n \in \mathbb{N}$. The fact that $\Phi(K)
\supset \mathbb{B}(R)$ follows since we are dealing with open
maps.

\subsection{Non-autonomous conjugation}

In the non-autonomous setting it is very rare that a single change of coordinates simplifies the sequence of maps. Instead we use a sequence of coordinate changes.

\begin{lemma}\label{lemma:conjugation}
Let $(f_n)$ be a sequence of automorphisms that satisfies the hypotheses of Conjecture \ref{conj:main}, and suppose that there exist uniformly bounded sequences $(g_n)$ and $(h_n)$, with $h_n = \mathrm{Id} + h.o.t.$, such that the diagram
\begin{equation}\label{diagram}
\begin{CD}
\mathbb{C}^m @>f_0>> \mathbb{C}^m @>f_1>> \mathbb{C}^m @>f_2>> \cdots\\
@VVh_0V @VVh_1V @VVh_2V\\
\mathbb{C}^m @>g_0>> \mathbb{C}^m @>g_1>> \mathbb{C}^m @>g_2>> \cdots
\end{CD}
\end{equation}
commutes as germs of order $k$. Then $\Omega_{(f_n)} \cong \Omega_{(g_n)}$.
\end{lemma}

If the maps $(g_n)$ are all lower triangular polynomials then one still has that  the basin of the sequence $(g_n)$ is equal to $\mathbb C^m$. This simple fact was used in \cite{Peters} to prove the following lemma, which for simplicity we state in the case $m = 2$.

\begin{lemma}\label{lemma:dominant}
Let $(f_n)$ be a sequence of automorphisms satisfying the conditions in Conjecture \ref{conj:main}, and suppose that the linear part of each map $f_n$ is of the form
$$
(z, w) \mapsto (a_n z, b_n w + c_n z),
$$
with $|b_n|^2 < \xi |a_n|$ for some uniform constant $\xi < 1$. Then we can find bounded sequences $(g_n)$ and $(h_n)$ as in Lemma \ref{lemma:conjugation}. Moreover, the maps $g_n$ can be chosen to be lower triangular polynomials, and hence $\Omega_{(f_n)} \cong \mathbb C^2$.
\end{lemma}

As was pointed out in \cite{Peters}, we can always find a non-autonomous change of coordinates such that the linear parts of the maps $f_n$ all become lower triangular. Let us explain the technique in the $2$-dimensional setting, where a matrix is lower diagonal if and only if $[0,1]$ is an eigenvector. Using $QR$-factorization we can find, for any vector $v_0$, a sequence of unitary matrices $(U_n)$ such that $U_0v_0=[0,1]$ and
$$
U_{n+1} \cdot D(f_n \circ \cdots \circ f_0)(0) v_0 = \lambda_n \cdot [0,1],
$$
with $\lambda_n \in \mathbb C$. Then by defining $g_n = U_{n+1} \circ f_n \circ U_n^{-1}$, we obtain a new sequence $(g_n)$ whose basin is equivalent (by the biholomorphic map $U_0$) to the basin of the sequence $(f_n)$. Notice that the linear parts of all the maps $(g_n)$ are lower triangular. In this construction we are free to choose the initial tangent vector $v_0$.

But while we may always assume that the linear parts are lower triangular, the condition that $|b_n|^2 < \xi |a_n|$ for all $n \in \mathbb N$ in Lemma \ref{lemma:dominant} is a strong assumption. In particular there is no reason to think that one can change coordinates to obtain a sequence of lower triangular polynomial maps. Instead we could aim for obtaining lower triangular polynomials on sufficiently large time-intervals. Let us be more precise. Suppose that one can find an increasing sequence of integers $p_1, p_2, \ldots$ and change coordinates as in Lemma \ref{lemma:conjugation} to obtain that $\Omega_{(f_n)}$ is equivalent to the basin of a sequence
\begin{equation} \label{eq:AM}
g_0, \ldots ,g_{p_1-1}, U_1, g_{p_1}, \ldots , g_{p_2 - 1}, U_2, \ldots,
\end{equation}
where the $U_j$ are unitary matrices and the maps $g_n$ are all lower triangular. In the spirit of Theorem \ref{thm:repeat} one would expect that the basin of this new sequence is equal to $\mathbb C^2$ as long as the sequence $(p_j)$ is sparse enough. Indeed this is the case, as follows from the following Lemma, proved by Abbondandolo and Majer in \cite{AM}.

\begin{lemma}\label{lemma:sparse}
Suppose that the maps $g_n$ in the sequence given in Equation \eqref{eq:AM} are uniformly bounded lower triangular polynomials of degree $k$, and that
\begin{equation}\label{eq:infinitesum}
\sum \frac{p_{j+1} - p_j}{k^j} = +\infty.
\end{equation}
Then the basin of the sequence in Equation \ref{eq:AM} is equal to $\mathbb C^2.$
\end{lemma}

The proof of Theorem \ref{thm:AM} from \cite{AM} can now be sketched as follows. Define $p_{j+1} = k^j + p_j$, and on each interval find a tangent vector $v_j$ which is contracted most rapidly by the maps $Df_{p_{j+1}, p_j}$. Next, find the non-autonomous change of coordinates by unitary matrices such that the maps $g_n$ as defined above all have lower triangular linear part. Then on each interval $I_j = [p_j, p_{j+1}]$ the maps $g_n$ satify the conditions of Lemma \ref{lemma:dominant} ``on average''. This is enough to find another non-autonomous change of coordinates after which the maps $g_n$ are lower triangular polynomial maps on each interval $I_j$. Then it follows from Lemma \ref{lemma:sparse} that the basin of the sequence $(f_n)$ is equivalent to $\mathbb C^2$.

Now we arrive at one of the main points presented in this article. In the argument of Abbondandolo and Majer the intervals $[p_j, p_{j+1}]$ are chosen without taking the maps $(f_n)$ into consideration; it is sufficient to make a simple choice such that Equation \eqref{eq:infinitesum} is satisfied. In the last section of this article we will show that we can obtain stronger results if we instead let the intervals $[p_j, p_{j+1}]$ depend on the maps $(f_n)$, or to be more precise, on the linear parts of the maps $(f_n)$.

\subsection{Abstract Basins}

Let us discuss a construction due to Forn{\ae}ss and Stens{\o}nes \cite{FS}. Let $(f_n)$ now be a sequence of biholomorphic maps from the unit ball $\mathbb B$ into $\mathbb B$, satisfying
$$
C \|z\| \le \|f_n(z)\| \le D\|z\|
$$
for some uniform $1>D>C>0$ as usual. We define the \emph{abstract basin of attraction} of the sequence $(f_n)$ as follows. Consider all sequences of the form
$$
(x_k, x_{k+1}, \ldots), \; \; \mathrm{with} \; \; x_{n+1} = f_n(x_n) \; \; \mathrm{for} \; \mathrm{all} \; n \ge k.
$$
We say that
$$
(x_k, x_{k+1}, \ldots) \sim (y_l, y_{l+1}, \ldots)
$$
if there exists a $j \ge \max(k,l)$ such that $x_j = y_j$. This gives an equivalence relation $\sim$, and we define
$$
\Omega_{(f_n)} = \left\{(x_k, x_{k+1}, \ldots) \mid f_n(x_n) = x_{n+1} \right\}/\sim
$$
We refer to $\Omega_{(f_n)}$ as the \emph{abstract basin of attraction}, sometimes also called the \emph{tail space}. We have now used the notation $\Omega_{(f_n)}$ for both abstract and non-autonomous basins, but thanks to the following lemma this will not cause any problems.

\begin{lemma}\label{lemma:abstract}
Let $(f_n)$ be a sequence of automorphisms of $\mathbb C^m$ which satisfy the conditions in Conjecture \ref{conj:main}. Then the basin of attraction of the sequence $(f_n)$ is equivalent to the abstract basin of the sequence $(f_n|_{\mathbb B})$.
\end{lemma}

Hence from now on we allow ourselves to be careless and write $\Omega_{(f_n)}$ for both kinds of attracting basins. Abstract basins were used by Forn{\ae}ss and Stens{\o}nes to prove the following.

\begin{thm}[Forn{\ae}ss-Stens{\o}nes]\label{thm:FS}
Let $f$ and $p$ be as in Conjecture \ref{conj:stable}. Then $\Sigma^s_f(p)$ is equivalent to a domain in $\mathbb C^m$.
\end{thm}

\begin{remark} Working with abstract basins can be very convenient. For example, Lemma \ref{lemma:conjugation} also holds for abstract basins which, in conjunction with Lemma \ref{lemma:abstract}, implies that in Diagram \ref{diagram} we do not need to worry about whether the maps $h_n$ and $g_n$ are globally defined automorphisms. From the fact that the sequences $(g_n)$ and $(h_n)$ are uniformly bounded it follows that their restrictions to some uniform neighborhood of the origin are biholomorphisms, which is all that is needed.
\end{remark}

\section{Definition of the trains}

Let us go back to the proof of \ref{thm:AM} by Abbondandolo and Majer. Instead of conjugating to lower triangular maps on the entire sequence $(f_n)$, they introduced a partition of $\mathbb N$ into intervals of rapidly increasing size, and on each interval changed coordinates to either lower or upper triangular maps. These intervals were referred to as \emph{trains}, and we will build on this terminology.

The main difference between the proof by Abbondandolo and Majer and the technique introduced in this paper is that we let the trains be determined by the sequence $(f_n)$, or rather, by the linear parts of the maps. We will describe an algorithm to construct the partition
$$
\mathbb N = \bigcup [p_j, p_{j+1}).
$$
Each train $[p_j, p_{j+1})$ will be headed by an interval $[p_j, q_j)$ which we will call the \emph{engine}. On the engine we will have very good estimates of the linear maps. Our estimates are not nearly as strong on the interval $[q_j, p_{j+1})$, but the good estimates on the engine will be used as a buffer to deal with estimates on the rest of the train. In fact, as soon as the buffer from the engine fails to be sufficient, we start a new train.

We proceed to define the trains explicitly, both in the diagonal and the general case.

\subsection{The diagonal case}

We assume here that the linear part of each map $f_n$ is of the form
$$
\left[
\begin{matrix}
a_n & 0\\
0 & b_n\\
\end{matrix}
\right]
$$
We then define
\begin{equation}
\begin{aligned}
f_{m,n} & = f_{m-1} \circ \cdots \circ f_n, \\
a_{m,n} & = a_{m-1} \cdot \cdots \cdot a_n, \; \; \mathrm{and} \\
b_{m,n} & = b_{m-1} \cdot \cdots \cdot b_n.
\end{aligned}
\end{equation}
We will define an increasing sequence $p_0, q_0, p_1, \ldots$ as follows. We set $p_0 = 0$ and $q_0 = 2$. We define the rest of the sequence recursively. Suppose that we have already defined $p_j$ and $q_j$. We consider all pairs $p_{j+1}, q_{j+1}$ satisfying $q_j \le p_{j+1} \le q_{j+1}$ and
\begin{equation}\label{eq:bound}
\left| \frac{a_{q_{j+1},p_{j+1}}}{b_{q_{j+1},p_{j+1}}}\right|^{(-1)^j} \ge D^{-k^{j+1}}.
\end{equation}
Out of all such pairs we choose $q_{j+1}$ minimal, and then $p_{j+1}$ such that
$$
\left| \frac{a_{q_{j+1},p_{j+1}}}{b_{q_{j+1},p_{j+1}}}\right|^{(-1)^j}
$$
is maximal. We will write $I_j = [p_j, p_{j+1})$ and refer to this interval as the $j$th-train. The interval $[p_j, q_j]$ we will call the \emph{engine} of the train $I_j$.

In the diagonal case we can immediately deduce several properties of the trains $I_j$. First of all, from equation \eqref{eq:bound} it is clear that
\begin{equation*}
\sum \frac{|I_j|}{k^j} = \infty,
\end{equation*}
which will later allow us to apply Lemma \ref{lemma:sparse}.

\begin{prop}\label{prop:trains}
On the engine of the train we have
\begin{equation}\label{eq:train1}
\left| \frac{a_{n,p_j}}{b_{n,p_j}}\right|^{(-1)^{j+1}} \ge 1 \; \; \mathrm{and} \; \; \left| \frac{a_{q_j,n}}{b_{q_j,n}}\right|^{(-1)^{j+1}} \ge 1,
\end{equation}
for $q_j \ge n \ge p_j$, and moreover
\begin{equation}\label{eq:train2}
\frac{D^{-k^j}}{C} \ge \left| \frac{a_{q_j,p_j}}{b_{q_j,p_j}}\right|^{(-1)^{j+1}} \ge D^{-k^j}.
\end{equation}
Furthermore, for $p_{j+1} \ge n \ge m \ge q_j$ we have
\begin{equation}\label{eq:train3}
\left| \frac{a_{n,m}}{b_{n,m}}\right|^{(-1)^{j+1}} > D^{k^{j+1}} \; \; \mathrm{and} \; \; \left| \frac{a_{p_{j+1},q_j}}{b_{p_{j+1},q_j}}\right|^{(-1)^{j+1}} \ge 1.
\end{equation}
\end{prop}

All three properties follow immediately from the definitions.

\medskip

\subsection{The general case}

Our goal is now to emulate the above construction while dropping the assumption that the linear parts of the maps $f_n$ are diagonal. In the diagonal case we eventually will obtain a sequence $(g_n)$ where each map $g_n$ is either lower triangular or upper triangular, depending on whether $n$ lies in an odd or even train. In the general case this will be generalized as follows: for each train $[p_j, p_{j+1})$ there will be a unit vector $v_j$ with respect to which the maps $g_n$ will be lower triangular. Let us be more precise. Suppose that the linear part of a map $g$ has an eigenvector $v$. Let $U$ be a $2\times 2$ unitary matrix such that
$$
Uv=
\left[
\begin{matrix}
0 \\
1
\end{matrix}
\right].
$$
We say that $g$ is lower triangular \emph{with respect to $v$} if
$U g U^{-1}$ is a lower triangular polynomial. Note that this
definition is independent of the choice of $U$.

Looking back, we see that in the diagonal case our maps $g_n$ are all lower triangular with respect to either $[0,1]$ or $[1,0]$. In the general case we allow the vector $v$ to be any unit vector in $\mathbb C^2$, but $v$ must remain constant on each train. This motivates the following definition.

\begin{defn}\label{def:trains}
For $k>x > 1$, recursively define an increasing sequence of
integers $p_1, q_1, p_2, q_2, \ldots$ and a sequence of
unit-vectors $v_1, v_2, \ldots$ as follows:

\medskip

Set $p_0=0$, $q_0=\lceil\frac{2}{k-x}\rceil$ and $v_0=(1,0)$.
Define the rest of the sequences by induction:

Given $p_j$, $q_j$ and $v_j$, consider all possible pairs of
integers $p_{j+1}$ and $q_{j+1}$ such that $q_{j+1} > p_{j+1} \ge
q_{j}$ and for which
$$
 \frac{|\mathrm{det} Df_{q_{j+1}, p_{j+1}}(0)|\cdot
|df_{p_{j+1}, p_j}(v_j)|^{x+1}}{|df_{q_{j+1}, p_j}(v_j)|^{x+1}}
\le D^{2^{j+1}}.
$$

Out of all possible choices for $p_{j+1}$ and $q_{j+1}$, we choose
$q_{j+1}$ minimal, and then $p_{j+1}$ such that $$\frac{|\mathrm{det} Df_{q_{j+1}, p_{j+1}}(0)|\cdot |df_{p_{j+1},
p_j}(v_j)|^{x+1}}{|df_{q_{j+1}, p_j}(v_j)|^{x+1}}$$ is minimal.

\medskip

By composing $Df_{q_{j+1},p_{j+1}}$ with a unitary matrix
$U_{j+1}$, we may assume that $df_{p_{j+1},p_j}(v_j)$ is an
eigenvector of $U_{j+1} \circ Df_{q_{j+1},p_{j+1}}$. Let $v_{j+1}$ be the other
unit-length eigenvector of $U_{j+1} \circ Df_{q_{j+1},p_{j+1}}$ (unique up to multiplication by a unit-length scalar). This vector must exist, since the eigenvector
$df_{p_{j+1},p_j}(v_j)$ has eigenvalue $\frac{|df_{q_{j+1},
p_j}(v_j)|}{|df_{p_{j+1}, p_j}(v_j)|}$ and we have
\begin{align*}
\left(\frac{|df_{q_{j+1}, p_j}(v_j)|}{|df_{p_{j+1},
p_j}(v_j)|}\right)^2 &\geq
\left(\frac{1}{D}\right)^{(2^{j+2})/(x+1)}\cdot |\mathrm{det}
Df_{q_{j+1}, p_{j+1}}(0)|^{2/(x+1)}\\
&> 1 \cdot |\mathrm{det} Df_{q_{j+1},
p_{j+1}}(0)|^{4/5}>|\mathrm{det} Df_{q_{j+1}, p_{j+1}}(0)|
\end{align*}
since $|\mathrm{det} Df_{q_{j+1},
p_{j+1}}(0)|\leq D^{2(q_{j+1}-p_{j+1})}<1$. This implies that
$U_{j+1} \circ Df_{q_{j+1},p_{j+1}}$ has two distinct eigenvalues,
and hence two distinct eigenvectors.
\end{defn}

Our strategy to prove inequalities like those in Proposition
\ref{prop:trains} on an engine $[p_j, q_j)$, is the following. We
will work not only with the diagonal entries $a_n$ and $b_n$ in
the current coordinates, but first with the entries $\alpha_n$ and
$\beta_n$, which are the diagonal entries in the coordinates that
were used for the previous train. By the recursive definition of
the train, we first obtain estimates for $\alpha_n$ and $\beta_n$,
and then translate these to estimates on $a_n$ and $b_n$ in the
new coordinates. In the translation from old to new coordinates
the off-diagonal entries, denoted by $c_n$ and $\gamma_n$, will
play an important role.

Given $(p_j)$, $(q_j)$ and $(v_j)$ as defined above, we now apply a non-autonomous conjugation such that within each train $[p_j, p_{j+1})$ all linear parts have $[0,1]$ as an eigenvector. We know that $v_j$ is an
eigenvector of $U_j \circ Df_{q_{j},p_{j}}(0)$. Pick a unitary
matrix $V_j$ that transfers $v_j$ into $[0,1]$. Then $[0,1]$ will
be eigenvector of $V_j \circ U_j \circ Df_{q_{j},p_{j}}(0) \circ
V_j^{-1}$. Next, for $i=p_j+1, \ldots p_{j+1}$, pick a unitary
matrix $W_i$ such that $[0,1]$ is an eigenvector of $W_i \circ
Df_{i,p_{j}}(0) \circ V_j^{-1}$. We can pick $W_{q_j}=V_j \circ
U_j$. Define
\begin{equation}
\tilde{f}_i=\begin{cases} W_{i+1} \circ f_i \circ W_i^{-1} &
\textrm{ if }p_j < i <p_{j+1},\\
W_{p_j+1} \circ f_{p_j} \circ V_j^{-1} & \textrm{ if } i=p_j.
\end{cases}
\end{equation}

Now the composition $f_{p_{j+1},p_j}=f_{p_{j+1}-1} \circ \cdots
\circ f_{p_j}$ can also be written as
\begin{equation}\label{eq:ftilde}
f_{p_{j+1},p_j}=W_{p_{j+1}}^{-1} \circ \tilde{f}_{p_{j+1}-1}
\circ \ldots \circ \tilde{f}_{p_j} \circ V_j.
\end{equation}
We know
that $[0,1]$ is an eigenvector of each $D\tilde{f}_{n,m}(0)$ for
$p_j \leq m \leq n \leq p_{j+1}$. So these linear maps are
lower-triangular: write
$D\tilde{f}_{n,m}(0)=\left(\begin{smallmatrix} a_{n,m}&0\\
c_{n,m}&b_{n,m}
\end{smallmatrix} \right)$.

Then we have
\begin{equation}\label{eq:b}
|b_{n,m}| =\frac{|d\tilde{f}_{n,
p_j}\left([0,1]\right)|}{|d\tilde{f}_{m, p_j}\left([0,1]\right)|}=
\frac{|df_{n, p_j}(v_j)|}{|df_{m, p_j}(v_j)|},
\end{equation}
and
\begin{equation}\label{eq:a}
|a_{n,m}| =|\mathrm{det} D\tilde{f}_{n,m}(0)| / |b_{n,m}|=
|\mathrm{det} Df_{n,m}(0)| / |b_{n,m}|.
\end{equation}
Write $a_n=a_{n+1,n}$, $b_n=b_{n+1,n}$ and $c_n=c_{n+1,n}$. Note that
$a_{n,m}=\prod_{i=m}^{n-1}a_i$ and $b_{n,m}=\prod_{i=m}^{n-1} b_i$ as in the diagonal case, and
$$
c_{n,m}=\sum_{i=m}^{n-1} b_{n,i+1} c_i a_{i,m}.
$$

Since all $V_j$ and $W_i$ are unitary matrices, we still
have that $C\|z\| \le \|\tilde{f}_n(z)\| \le D\|z\|$ for $z \in \mathbb B$, and hence we
know that $C \leq a_i \leq D$, $C \leq b_i \leq D$ and $c_i \leq
D$ for all $i$.

\begin{lemma}\label{lemma:trein}
In this new notation, our trains have the following properties:
\begin{enumerate}
\item[(i)]{On the engine of train $j$ we know that
$$\left|\frac{a_{q_j,p_j}^x}{b_{q_j,p_j}}\right| \ge D^{-2^{j}}.$$}
\item[(ii)]{On the wagons of the train we have:
\begin{equation*}
\left|\frac{a_{n,m}}{b_{n,m}^x}\right|> D^{2^{j+1}}, \; \; \mathrm{and} \; \; \left|\frac{a_{p_{j+1},q_j}}{b_{p_{j+1},q_j}^x}\right|\ge 1.
\end{equation*}
for $q_j \leq m \leq n \leq p_{j+1}$.}
\item[(iii)]{And for all $j \geq 0$ we have an estimate for the length of the $j$'th engine: $$q_j-p_j
\geq \frac{2^j}{k-x}.$$}
\end{enumerate}
\end{lemma}
\begin{proof}
To prove the first statement, note that by (\ref{eq:b}),
(\ref{eq:a}) and the fact that $U_j$ is unitary we have
$$
\left|\frac{a_{q_j,p_j}^x}{b_{q_j,p_j}}\right|=\frac{|\mathrm{det} Df_{q_j,p_j}(0)|^x}{|b_{q_j,p_j}|^{x+1}}=\frac{|\mathrm{det} Df_{q_j,p_j}(0)|^x}{|df_{q_j, p_j}(v_j)|^{x+1}}=\frac{|\mathrm{det} (U_j \circ Df_{q_j,p_j}(0))|^x}{|(U_j \circ df_{q_j, p_j})(v_j)|^{x+1}}.
$$
Since $v_j$ and $df_{p_{j},p_{j-1}}(v_{j-1})$ are distinct
eigenvectors of $U_j \circ Df_{q_j,p_j}(0)$, we obtain
\begin{align*}
\frac{|\mathrm{det} (U_j \circ Df_{q_j,p_j}(0))|^x}{|(U_j \circ
df_{q_j, p_j})(v_j)|^{x+1}}
= \frac{|df_{q_j, p_{j-1}}(v_{j-1})|^{x+1}}{|\mathrm{det}
Df_{q_j,p_j}(0)|\cdot |df_{p_j, p_{j-1}}(v_{j-1})|^{x+1}}.
\end{align*}
And by definition of $q_j$ and $p_j$, we
know that
$$
\frac{|df_{q_j,
p_{j-1}}(v_{j-1})|^{x+1}}{|\mathrm{det} Df_{q_j,p_j}(0)|\cdot
|df_{p_j, p_{j-1}}(v_{j-1})|^{x+1}} \geq D^{-2^j},
$$
which completes (i).

For $q_j \leq
m\leq n \leq p_{j+1}$ we have:
$$\left|\frac{a_{n,m}}{b_{n,m}^x}\right|=\frac{|\mathrm{det} Df_{n, m}(0)|\cdot |df_{m, p_j}(v_j)|^{x+1}}{|df_{n, p_j}(v_j)|^{x+1}}> D^{2^{j+1}}$$
by minimality of $q_{j+1}$. The second half of statement (ii) follows from our choice of $p_{j+1}$.

\medskip

Finally, use part $(i)$ to see that $$D^{-2^j} \leq
\left|\frac{a_{q_j,p_j}^x}{b_{q_j,p_j}}\right| \leq
\frac{(D^{q_j-p_j})^x}{C^{q_j-p_j}}\leq
\left(\frac{D^x}{D^k}\right)^{q_j-p_j}=D^{-(k-x)(q_j-p_j)},$$
which implies that $q_j-p_j \geq \frac{2^j}{k-x}$. The case $j=0$ holds by definition.
\end{proof}

Since $f_{p_{j},p_{j-1}}(v_{j-1})$ is another eigenvector of
$U_{j} \circ Df_{q_{j},p_{j}}$, we now introduce notation that
emphasizes the role of this vector. Pick a unitary matrix $S_j$
which maps $f_{p_{j},p_{j-1}}(v_{j-1})$ to $[0,1]$. Then pick
unitary matrices $T_i$ for $i=p_j+1, \ldots q_{j}$ such that
$[0,1]$ is an eigenvector of $T_i \circ Df_{i,p_{j}} \circ
S_j^{-1}$. Note that we can take $T_{q_j}=S_j \circ U_j$.

As before, define
\begin{equation}
\hat{f}_i=\begin{cases} T_{i+1} \circ f_i \circ T_i^{-1} &
\textrm{ if }p_j < i <q_j\\
T_{p_j+1} \circ f_{p_j} \circ S_j^{-1} & \textrm{ if } i=p_j
\end{cases}
\end{equation}
Then we have $f_{q_j,p_j}=W_{q_j}^{-1} \circ \tilde{f}_{q_j,p_j}
\circ V_j=T_{q_j}^{-1} \circ \hat{f}_{q_j,p_j} \circ S_j$.

We know that $[0,1]$ is an eigenvector of each $D\hat{f}_{n,m}(0)$
for $p_j \leq m \leq n \leq q_j$. Therefore we can write
$D\hat{f}_{n,m}(0)=\left(\begin{smallmatrix} \alpha_{n,m}&0\\
\gamma_{n,m}&\beta_{n,m}
\end{smallmatrix} \right)$.

Then we have for $p_j \leq m \leq n \leq q_j$:
$$
|\beta_{n,m}| = \frac{|df_{n,
p_{j}}(f_{p_{j},p_{j-1}}(v_{j-1}))|}{|df_{m,
p_{j}}(f_{p_{j},p_{j-1}}(v_{j-1}))|}=\frac{|df_{n,
p_{j-1}}(v_{j-1})|}{|df_{m, p_{j-1}}(v_{j-1})|},
$$
and
$$
|\alpha_{n,m}| = |\mathrm{det} Df_{n,m}| / |\beta_{n,m}|.
$$
Writing $\alpha_n=\alpha_{n+1,n}$,
$\beta_n=\beta_{n+1,n}$ and $\gamma_n=\gamma_{n+1,n}$ as before, we again have
$$
\gamma_{n,m}=\sum_{i=m}^{n-1} \beta_{n,i+1} \gamma_i
\alpha_{i,m},
$$
and
$$
\gamma_i \leq D
$$
for all $i$.

Note that $|a_n \cdot b_n| = |\mathrm{det} Df_n |=
|\alpha_n \cdot \beta_n|$ for all $n$. Since
$\beta_{q_{j},p_{j}}$ and $b_{q_{j},p_{j}}$ are the eigenvalues of
$U \circ Df_{q_{j},p_{j}}$ for two distinct eigenvectors, we must have $\beta_{q_{j},p_{j}} \cdot
b_{q_{j},p_{j}}=\mathrm{det} Df_{q_{j},p_{j}}$, which implies that
$|\beta_{q_{j},p_{j}}|=|a_{q_{j},p_{j}}|$ and
$|\alpha_{q_{j},p_{j}}|=|b_{q_{j},p_{j}}|$.

This new notation allows us to derive additional inequalities for the engines of the trains:
\begin{lemma}\label{lemma:Grieks}
For $p_j \leq m \leq n \leq q_j$ we have:
\begin{enumerate}
\item[(i)]{$\left|\frac{\alpha_{q_j, n}}{\beta_{q_j, n}^x}\right| \leq 1$,}
\item[(ii)]{$\left|\frac{\alpha_{n,p_j}}{\beta_{n,p_j}^x}\right| \leq
1$, and}
\item[(iii)]{$D^{2^j}D^{k-x}< \left|\frac{\alpha_{n,m}}{\beta_{n,m}^x}\right|<D^{-2^j}$.}
\end{enumerate}
\end{lemma}

\begin{proof}
By the choice of $p_j$ and $q_j$ and the fact that
\begin{align*}
\frac{|\mathrm{det} Df_{q_{j},p_{j} }(0)|\cdot
|df_{p_{j}, p_{j-1}}(v_{j-1})|^{x+1}}{|df_{q_{j},
p_{j-1}}(v_{j-1})|^{x+1}} =& \frac{|\mathrm{det}
Df_{q_{j}, n}(0)|\cdot |df_{n,
p_{j-1}}(v_{j-1})|^{x+1}}{|df_{q_{j}, p_{j-1}}(v_{j-1})|^{x+1}}\\
&\cdot \frac{|\mathrm{det} Df_{n, p_{j}}(0)|\cdot
|df_{p_{j}, p_{j-1}}(v_{j-1})|^{x+1}}{|df_{n,
p_{j-1}}(v_{j-1})|^{x+1}},
\end{align*}
we obtain (i) and (ii):
$$
\left|\frac{\alpha_{q_j, n}}{\beta_{q_j,
n}^x}\right|= \frac{|\mathrm{det} Df_{q_{j}, n}(0)|\cdot
|df_{n, p_{j-1}}(v_{j-1})|^{x+1}}{|df_{q_{j},
p_{j-1}}(v_{j-1})|^{x+1}} \leq 1
$$
and
$$
\left|\frac{\alpha_{n, p_j}}{\beta_{n,
p_j}^x}\right|=\frac{|\mathrm{det} Df_{n, p_{j}}(0)|\cdot
|df_{p_{j}, p_{j-1}}(v_{j-1})|^{x+1}}{|df_{n,
p_{j-1}}(v_{j-1})|^{x+1}} \leq 1.
$$

Inequality (iii) is trivial if $n=m$. If $m<n$ we have:
\begin{align*}
\left|\frac{\alpha_{n,m}}{\beta_{n,m}^x}\right|&=\left|
\frac{\alpha_{n,p_j}}{\beta_{n,p_j}^x}\right|\cdot \left|
\frac{\alpha_{m,p_j}}{\beta_{m,p_j}^x}\right|^{-1}\\
&\leq 1 \cdot \left(\frac{|\mathrm{det} Df_{m,
p_{j}}(0)|\cdot |df_{p_{j}, p_{j-1}}(v_{j-1})|^{x+1}}{|df_{m,
p_{j-1}}(v_{j-1})|^{x+1}}\right)^{-1}<\left(D^{2^j}\right)^{-1}=D^{-2^j}
\end{align*}
since $m<q_j$, which proves the right-hand side of inequality (iii).

By definition of $q_j$ we have
$$
\left|\frac{\alpha_{n-1,m}}{\beta_{n-1,m}^x}\right|=
\frac{|\mathrm{det} Df_{n-1, m}(0)|\cdot |df_{m,
p_{j-1}}(v_{j-1})|^{x+1}}{|df_{n-1, p_{j-1}}(v_{j-1})|^{x+1}} \geq
D^{2^j},
$$
since $n-1<q_j$. Hence
$$
\left|\frac{\alpha_{n,m}}{\beta_{n,m}^x}\right|=
 \left|\frac{\alpha_{n-1,m}}{\beta_{n-1,m}^x}\right|\cdot
\left|\frac{\alpha_{n-1}}{\beta_{n-1}^x}\right|\geq D^{2^j}\cdot \frac{C}{D^x}
> D^{2^j} \cdot D^{k-x},
$$
which implies the left-hand side of inequality (iii).
\end{proof}

We can now give an upper estimate for
$|\gamma_{n,m}|$ on the engine as well:

\begin{lemma}\label{lemma:gamma}
For $p_j \leq m \leq n \leq q_j$, we have
$$
|\gamma_{n,m}| \leq
K_1 \cdot
\left|\frac{\beta_{n,p_j}}{\alpha_{m,p_j}^{1/x}}\right|,
$$
where $K_1=\frac{D}{C(1-D^{1-1/x})}$ is a constant depending only on $C$, $D$ and $x$.
\end{lemma}

\begin{proof}
We already know that $|\gamma_i|<D$ for all $i=p_j, \ldots,
q_j-1$, and we have:
\begin{align*}
|\gamma_{n,m}| & \leq \sum_{i=m}^{n-1} |\beta_{n,i+1}\gamma_i \alpha_{i,m}|
= \sum_{i=m}^{n-1} \frac{|\alpha_{i,p_j}|}{|\beta_{i,p_j}|} \cdot \frac{|\beta_{i,p_j}| \cdot |\beta_{n,i+1}| \cdot |\gamma_i| \cdot |\alpha_{i,m}|}{|\alpha_{i,p_j}|}\\
&\leq \sum_{i=m}^{n-1} \frac{|\alpha_{i,p_j}|}{|\beta_{i,p_j}|} \cdot \frac{1}{|\alpha_{m,p_j}|} \cdot D \cdot \frac{|\beta_{n,p_j}|}{|\beta_i|} \leq \frac{D}{C} \cdot \frac{|\beta_{n,p_j}|}{|\alpha_{m,p_j}|}
\sum_{i=m}^{n-1} \frac{|\alpha_{i,p_j}|}{|\beta_{i,p_j}|},
\end{align*}
where we used that $|\gamma_i|<D$ and $|\beta_i|>C$.

By Lemma \ref{lemma:Grieks} we have
$\left|\frac{\alpha_{i,p_j}}{\beta_{i,p_j}^x}\right| \leq 1$, so
$|\alpha_{i,p_j}| \leq |\beta_{i,p_j}|^x$ and
$|\alpha_{i,p_j}|^{1/x} \leq |\beta_{i,p_j}|$. Hence we obtain
\begin{align*}
|\gamma_{n,m}| & \le \frac{D}{C} \cdot
\frac{|\beta_{n,p_j}|}{|\alpha_{m,p_j}|} \sum_{i=m}^{n-1}
\frac{|\alpha_{i,p_j}|}{|\beta_{i,p_j}|}
&& \leq \frac{D}{C} \cdot \frac{|\beta_{n,p_j}|}{|\alpha_{m,p_j}|}
\sum_{i=m}^{n-1} |\alpha_{i,p_j}|^{1-1/x}\\
&= \frac{D}{C} \cdot
\frac{|\beta_{n,p_j}|}{|\alpha_{m,p_j}|^{1/x}} \sum_{i=m}^{n-1}
|\alpha_{i,m}|^{1-1/x}
&&\leq \frac{D}{C} \cdot
\frac{|\beta_{n,p_j}|}{|\alpha_{m,p_j}|^{1/x}} \sum_{i=m}^{n-1}
(D^{1-1/x})^{i-m}\\
&\leq \frac{D}{C} \cdot
\frac{|\beta_{n,p_j}|}{|\alpha_{m,p_j}|^{1/x}}
\cdot \frac{1}{1-D^{1-1/x}}
\end{align*}
\end{proof}

Using Lemmas \ref{lemma:Grieks} and \ref{lemma:gamma}, we can now find estimates for
$|df_{n,p_j}(v_j)|$ for $p_j \leq n \leq q_j$, which we can use to
translate our information on $\alpha$ and $\beta$ to information
on $a$ and $b$.

\begin{lemma}\label{lemma:vector}
For $p_j \leq n \leq q_j$, we have
\begin{enumerate}
\item[(i)]{$|df_{n,p_j}(v_j)|  \leq 2 \max\left\{|\alpha_{n,p_j}|,
K_1 \cdot |\alpha_{n,p_j}|^{1-1/x}|\beta_{n,p_j}|\right\}$, and}
\item[(ii)]{$\frac{1}{|df_{n,p_j}(v_j)|}
 \leq K_1 \cdot \frac{1}{|\alpha_{n,p_j}|}$.}
 \end{enumerate}

\end{lemma}

\begin{proof}
We have
\begin{align*}
|df_{n, p_j}(v_j)|&=|(U_j \circ df_{q_j,n})^{-1}\left( U_j \circ df_{q_j,p_j}v_j \right)|=|(U_j \circ df_{q_j,n})^{-1}\left(b_{q_j,p_j} v_j \right)|\\
&\leq |b_{q_j,p_j}| \cdot ||(U_j \circ
df_{q_j,n})^{-1}||=|\alpha_{q_j,p_j}|\cdot ||\left(S_j \circ U_j
\circ
df_{q_j,n}\circ T_n^{-1}\right)^{-1}||\\
&=|\alpha_{q_j,p_j}|\cdot \| \left(
\begin{matrix} \alpha_{q_j,n}&0\\
\gamma_{q_j,n}& \beta_{q_j,n}
\end{matrix}
\right)^{-1}\|\\
& \leq |\alpha_{q_j,p_j}|\cdot 2
\max\left\{\frac{1}{|\alpha_{q_j,n}|},\frac{|\gamma_{q_j,n}|}{|\alpha_{q_j,n}\beta_{q_j,n}|},\frac{1}{|\beta_{q_j,n}|}\right\}.
\end{align*}
Using Lemma \ref{lemma:gamma} we can now calculate that
$$
\frac{|\gamma_{q_j,n}|}{|\alpha_{q_j,n}|\cdot |\beta_{q_j,n}|} \leq K_1
\cdot
\frac{|\alpha_{n,p_j}|^{1-1/x}|\beta_{n,p_j}|}{|\alpha_{q_j,p_j}|}.
$$
Since $\left|\frac{\alpha_{q_j,n}}{\beta_{q_j,n}^x}\right| \leq 1$ (by Lemma
\ref{lemma:Grieks}), we have $|\alpha_{q_j,n}| \leq
|\beta_{q_j,n}|^x \leq |\beta_{q_j,n}|$ and therefore
$\frac{1}{|\alpha_{q_j,n}|} \geq \frac{1}{|\beta_{q_j,n}|}$. Hence we obtain
\begin{align*}
|df_{n, p_j}(v_j)|&\leq |\alpha_{q_j,p_j}| \cdot 2 \max\left\{\frac{1}{|\alpha_{q_j,n}|}, K_1 \cdot \frac{|\alpha_{n,p_j}|^{1-1/x}|\beta_{n,p_j}|}{|\alpha_{q_j,p_j}|} \right\}\\
&= 2 \max\left\{|\alpha_{n,p_j}|, K_1 \cdot
|\alpha_{n,p_j}|^{1-1/x}|\beta_{n,p_j}|\right\}.
\end{align*}

To prove inequality (ii), note that
\begin{align*}
\frac{1}{|df_{n,p_j}v_j|}&=\frac{|(df_{n,p_j})^{-1}(df_{n,p_j}v_j)|}{|df_{n,p_j}v_j|}\leq ||df_{n,p_j}(0)^{-1}||\\
& =||\left(T_n \circ df_{n,p_j}(0) \circ S_j^{-1}\right)^{-1}||
 \leq 2 \max\left\{\frac{1}{|\alpha_{n,p_j}|},\frac{1}{|\beta_{n,p_j}|},\frac{|\gamma_{n,p_j}|}{|\alpha_{n,p_j}\beta_{n,p_j}|}\right\}.
\end{align*}

By Lemma \ref{lemma:Grieks} we have that $
\left|\frac{\alpha_{n,p_j}}{\beta_{n,p_j}^x}\right| \leq 1$, which implies that we may remove $\frac{1}{|\beta_{n,p_j}|}$ from
the maximum. Now Lemma \ref{lemma:gamma} implies that
\begin{align*}
\frac{1}{|df_{n,p_j}v_j|} &\leq \max\left\{\frac{1}{|\alpha_{n,p_j}|},\frac{1}{|\alpha_{n,p_j}\beta_{n,p_j}|} \cdot K_1\cdot |\beta_{n,p_j}|\right\}\\
&=K_1 \cdot \frac{1}{|\alpha_{n,p_j}|}.
\end{align*}
\end{proof}

We can now translate lemma \ref{lemma:Grieks} to the following generalization of Proposition \ref{prop:trains}.

\begin{thm}\label{thm:trein}
There exists a constant $K_2=K_2(C,D,k,x)$
such that for each $p_j \le m \le n \le q_j$ we have
\begin{enumerate}
\item[(i)]{$\left|\frac{b_{n, p_j}}{a_{n, p_j}^x}\right| \leq K_2 D^{-\frac{k-x}{x}(n-p_j)} $,}
\item[(ii)]{$\left|\frac{b_{n, m}}{a_{n, m}^x}\right| \leq K_2 D^{-2^j\frac{x+1}{x}} $, and}
\item[(iii)]{$\left|\frac{b_{q_j, n}}{a_{q_j, n}^x}\right| \leq K_2$.}
\end{enumerate}
\end{thm}

\begin{proof}
For part (i) we can use Lemma \ref{lemma:vector} and Lemma \ref{lemma:Grieks} to see that
\begin{align*}
\left| \frac{b_{n,p_j}}{a_{n,p_j}^x}\right| &= \frac{|df_{n,p_j}(v_j)|^{x+1}}{|\textrm{det}Df_{n,p_j}(0)|^x}\\
& \leq \frac{2^{x+1}}{|\alpha_{n,p_j}\beta_{n,p_j}|^x} \max\left\{|\alpha_{n,p_j}|,K_1 |\alpha_{n,p_j}|^{1-1/x}|\beta_{n,p_j}|\right\}^{x+1}\\
&= 2^{x+1} \cdot \max\left\{\frac{|\alpha_{n,p_j}|}{|\beta_{n,p_j}|^x},K_1^{x+1}\frac{|\beta_{n,p_j}|}{|\alpha_{n,p_j}|^{1/x}}\right\}\\
& \leq 2^{x+1} K_1^{x+1}D^{-\frac{k-x}{x}(n-p_j)}.
\end{align*}
Part (ii) and (iii) follow from the same lemmas:

\begin{align*}
\left| \frac{b_{n,m}}{a_{n,m}^x}\right| &= \frac{|df_{n,p_j}(v_j)|^{x+1}}{|\textrm{det}Df_{n,m}(0)|^x\cdot |df_{m,p_j}(v_j)|^{x+1}}\\
& \leq \frac{2^{x+1}}{|\alpha_{n,m}\beta_{n,m}|^x} \max\left\{|\alpha_{n,p_j}|,K_1|\alpha_{n,p_j}|^{1-1/x}|\beta_{n,p_j}| \right\}^{x+1} \left( \frac{K_1}{|\alpha_{n,p_j}|}\right)^{x+1}\\
& = (2K_1)^{x+1} \max\left\{\frac{|\alpha_{n,m}|}{\beta_{n,m}|^x},K_1^{x+1}\frac{|\beta_{n,p_j}|}{|\alpha_{n,p_j}|^{1/x}}\frac{|\beta_{m,p_j}|^x}{|\alpha_{m,p_j}|} \right\}\\
& \leq (2K_1)^{x+1} \max\left\{D^{-2^j},K_1^{x+1} \left(D^{2^j+(k-x)} \right)^{-1/x}\left(D^{2^j+(k-x)} \right)^{-1} \right\}\\
& \leq (2K_1^2)^{x+1} D^{-(k-x)\frac{x+1}{x}} D^{-2^j \frac{x+1}{x}}, \textrm{ and}\\
\left| \frac{b_{q_j,n}}{a_{q_j,n}^x}\right| &= \frac{|df_{q_j,p_j}(v_j)|^{x+1}}{|\textrm{det}Df_{q_j,n}(0)|^x\cdot |df_{n,p_j}(v_j)|^{x+1}}\\
& \leq \frac{|\alpha_{q_j,p_j}|^{x+1}}{|\alpha_{q_j,n}\beta_{q_j,n}|^x} \cdot K_1^{x+1} \frac{1}{|\alpha_{n,p_j}|^{x+1}}\\
& = K_1^{x+1} \frac{|\alpha_{q_j,n}|}{|\beta_{q_j,n}|^x} \leq K_1^{x+1}.
\end{align*}
So we can set $K_2=(2K_1^2)^{x+1} D^{-(k-x)\frac{x+1}{x}}$.
\end{proof}

\medskip

\begin{remark}
When we continue the proof of the general case we will work with the sequence the sequence $\tilde{f}_1,
\tilde{f}_2,\ldots \tilde{f}_{p_1-1}, V_1 \circ W_{p_1}^{-1},
\tilde{f}_{p_1}, \ldots, \tilde{f}_{p_2-1}, V_2 \circ
W_{p_2}^{-1}, \tilde{f}_{p_2}, \ldots$ instead of $(f_n)$. The basins of
these two sequences must be equivalent since at any time the composed functions differ only by multiplication
with one last unitary matrix.

Hence we will write $f_n$ rather than $\tilde{f}_n$ and
$M_j=V_j \circ W_{p_j}^{-1}$. We will study the basin of
the sequence $f_0,f_2, \ldots, f_{p_1-1},M_1,f_{p_1}, \ldots$
where $Df_n(0)=\left(\begin{smallmatrix} a_{n}&0\\
c_{n}&b_{n}
\end{smallmatrix} \right)$ and all $M_j$ are unitary.
\end{remark}

\section{Completion of the proof in the diagonal case}

Let us recall the statement of Theorem \ref{thm:main}.

\medskip

{\bf Theorem \ref{thm:main}.}
\emph{Let $(f_n)$ be a sequence of automorphisms of $\mathbb C^2$ which satisfies the conditions in Conjecture \ref{conj:main}, and suppose that the maps $f_n$ all have order of contact $k$. Assume further that the linear part of each map $f_n$ is diagonal. Then if $D^{k+1} < C$, the basin of attraction $\Omega_{(f_n)}$ is biholomorphic to $\mathbb C^2$.}

\medskip

The proof will be completed in two steps. In the first step we will find a non-autonomous conjugation on each train by linear maps such that in each train the new maps all contract the same coordinate most rapidly. This means that we can apply Lemma \ref{lemma:dominant2} below on each train separately. In the second step we worry about what happens at times $p_j$ where we switch from one train to the other. After these two steps we construct the maps $(g_n)$ and $(h_n)$ on each of the trains using the following Lemma, a special case of Lemma \ref{lemma:dominant}.

\begin{lemma}\label{lemma:dominant2}
Suppose that each map of the sequence $(f_n)$ has linear part of the form $(z, w) \mapsto (a_n z, b_n w)$, with $|b_n| \le |a_n|$. Then for any $k \ge 2$ we can find bounded sequences $(g_n)$ and $(h_n)$, with the maps $g_n$ lower triangular, such that Diagram \eqref{diagram} commutes up to jets of order $k$.
\end{lemma}

\medskip

\noindent {\bf Step 1:} Directing the trains.

\medskip

We assume that the automorphisms $f_0, f_1, \ldots$ satisfy the conditions of Theorem \ref{thm:main}, and that we have constructed the increasing sequence $p_0, q_0, p_1, \ldots$ such that Proposition \ref{prop:trains} holds. We will change coordinates with a sequence of linear maps $l_n$ of the form
\begin{equation}
l_n(z,w) = (\theta_n z, \tau_n w).
\end{equation}
We write $\tilde{f}_n = l_{n+1} \circ f_n\circ  l_n^{-1}$ for the new maps, and we immediately see that $\tilde{f}_n$ is of the form
\begin{equation}
\tilde{f}_n (z,w) = (\tilde{a}_n z, \tilde{b}_n w) + O(k),
\end{equation}
where
\begin{equation}
\begin{aligned}
\tilde{a}_n & = \frac{\theta_{n+1}}{\theta_n} a_n, \; \; \mathrm{and}\\
\tilde{b}_n & = \frac{\tau_{n+1}}{\tau_n} b_n.
\end{aligned}
\end{equation}

Our goal is to define the maps $(l_n)$ in such a way that on each interval $I_j = [p_j, p_{j+1})$ the maps $\tilde{f}_n$ satisfy the property $\tilde{a}_n \ge \tilde{b}_n$ (if $j$ is odd), and $\tilde{a}_n \le \tilde{b}_n$ (if $j$ is even). In order for the higher order terms of the maps $\tilde{f}_n$ not to blow up we also require that
\begin{equation}\label{boundeddistortion}
\theta_n^k \ge \tau_n, \; \; \mathrm{and} \; \theta_n \le \tau_n^k.
\end{equation}

In order to simplify the notation we let $j$ be odd; the results are analogous for $j$ even. We assume that
$$
\theta_p \ge D^{-\frac{k}{k^2 - 1} k^j}
$$
and
$$
\tau_p \ge D^{-\frac{1}{k^2 - 1} k^j}
$$
and that the conditions in \eqref{boundeddistortion} are satisfied for $n = p$. Then for $n+1 \in I_j$ we recursively define
\begin{equation}
\theta_{n+1} = \left\{ \begin{aligned}
\theta_n \; \; \mathrm{if} \; |a_n| \ge |b_n|, \; \mathrm{or} \\
\frac{|b_n|}{|a_n|} \theta_n \; \; \mathrm{if} \; |b_n| > |a_n|.
\end{aligned}\right.
\end{equation}
Similarly we define
\begin{equation}
\tau_{n+1} = \left\{ \begin{aligned}
& \tau_n \; \; \; \; \; \; \; \; \mathrm{if} \; |b_n| \ge |a_n|, \; \mathrm{or} \\
& \min(\frac{|a_n|}{|b_n|} \tau_n, \theta_{n+1}^k) \; \; \mathrm{if} \; |a_n| > |b_n|.
\end{aligned}\right.
\end{equation}
For convenience we will often write
$$
\sigma_{n,m} = ^{D^{-1}}\log\left|\frac{a_{n,m}}{b_{n,m}}\right|.
$$
\begin{lemma}
For $p \le n \le q$ we have
\begin{equation}
\tau_n^k \ge D^{-k \cdot \sigma_{n, p}} \theta_n.
\end{equation}
\end{lemma}
\begin{proof}
Follows easily by induction on $n$.
\end{proof}

Note that it follows immediately from the recursive definition of $\tau_{n+1}$ and the fact that $\theta_n$ can never decrease, that $\tau_n \le \theta_n^k$ for all $n \in [p, r]$.

\begin{lemma}
For $p \le n \le r$ we have that $\theta_n \le \tau_n^k$.
\end{lemma}
\begin{proof}
For $n \le q$ this is guaranteed by the previous lemma. From $q$ on it follows by induction on $n$ that
\begin{equation}
\frac{\tau_n^k}{\theta_n} \ge D^{-k^{j+1} - \sigma_{n,m}},
\end{equation}
for all $q \le m \le n$ for which $\sigma_{n,m} \le 0$. The statement in the Lemma now follows from \eqref{eq:train3}.
\end{proof}

Finally, we need to check that $\theta_r$ and $\tau_r$ are large enough to satisfy the starting hypothesis for the next interval $I_{j+1}$.

\begin{lemma}
We have $\theta_r \ge D^{-\frac{1}{k^2 - 1} k^{j+1}}$ and $\tau_p \ge D^{-\frac{k}{k^2 - 1} k^{j+1}}$.
\end{lemma}
\begin{proof}
The estimate on $\theta_r$ is immediate since $\theta_n$ does not decrease with $n$. The estimate on $\tau_p$ follows from
\begin{equation}
\tau_p \ge \tau_q \ge D^{-\sigma_{q, p}} \cdot D^{-\frac{k}{k^2 - 1} k^j}.
\end{equation}
\end{proof}

Our conclusion is the following:

\begin{thm}\label{thm:directing}
Let $(\tilde{f}_n)$ be the sequence defined by
\begin{equation}
\tilde{f}_n = l_{n+1} \circ f_n \circ l_n^{-1}.
\end{equation}
Then $(\tilde{f}_n)$ is a bounded sequence of automorphisms, and $\Omega_{(\tilde{f}_n)} \cong \Omega_{(f_n)}$. Moreover
$$
(-1)^{j+1} \log\left(\frac{|\tilde{a}_n|}{|\tilde{b}_n|}\right) \ge 0
$$
for all $n \in I_j$.
\end{thm}
\begin{proof}
The condition on the coefficients of the linear parts of the maps $\tilde{f}_n$ follows from the discussion earlier in this section. The fact that the sequence $(\tilde{f}_n)$ is bounded follows from the facts that the linear parts stay bounded, and that the higher order terms grow by at most a uniform constant.

To see that the two basins of attraction are biholomorphically equivalent (with biholomorphism $l_0$), note that $\tilde{f}_{n,0} = l_{n+1} \circ f_{n,0}\circ  l_0^{-1}$. Since the entries of the diagonal linear maps $l_{n+1}$ are always strictly greater than $1$, it follows that the basin of the sequence $(\tilde{f}_n)$ is contained in the $l_0$-image of the basin of the sequence $(f_n)$. The other direction follows from the fact that the coefficients of the maps $l_n$ grow at a strictly lower rate than the rate at which orbits are contracted to the origin by the sequence $(f_n)$.
\end{proof}

\medskip

\noindent{\bf Step 2:} Connecting the trains.

\medskip

The following is a direct consequence of Lemma \ref{lemma:dominant}:

\begin{thm}\label{thm:triangulization}
Let $(f_n)$ be a bounded sequence of automorphisms of $\mathbb{C}^2$ whose linear parts are of the form $(z,w) \mapsto (a_n z, b_n w)$ and whose order of contact is $k$. Suppose that $|a_n| \ge |b_n|$ for all $n = 0, 1, \ldots$. Then there exists bounded sequences $(h_n)$ and $(g_n)$ such that $g_n \circ h_n = h_{n+1} \circ f_n + O(k+1)$ for all $n$. Here, the maps $g_n$ can be chosen to be lower triangular maps, and the maps $h_n$ can be chosen to be of the form $(z,w) \mapsto \mathrm{Id} + h.o.t.$.
\end{thm}

In fact, even without the condition $|a_n| \ge |b_n|$ we can always change coordinates such that the maps $f_n$ become of the form
\begin{equation}
f_n(z,w) = (a_n z + c_n w^k, b_n w + d_n z^k) + O(k+1).
\end{equation}
Hence we may assume that our maps are all of this form. We outline the proof of Theorem \ref{thm:triangulization}. Our goal is to find sequences $(g_n)$ and $(h_n)$ such that the following diagram commutes up to degree $k$.
\begin{equation}
\begin{CD}
\mathbb{C}^2 @>f_0>> \mathbb{C}^2 @>f_1>> \mathbb{C}^2 @>f_2>> \cdots\\
@VVh_0V @VVh_1V @VVh_2V\\
\mathbb{C}^2 @>g_0>> \mathbb{C}^2 @>g_1>> \mathbb{C}^2 @>g_2>> \cdots
\end{CD}
\end{equation}
Here, $h_n$ will be of the form
\begin{equation} \label{form:h_n}
h_n: (z,w) \mapsto (z + \alpha_n w^k, w),
\end{equation}
and $g_n$ will be of the form
\begin{equation}
g_n: (z,w) \mapsto (a_n z, b_n w + \gamma_n z^k).
\end{equation}
We need that
\begin{equation}
h_n = g_n^{-1} \circ h_{n+1} \circ f_n + O(k+1).
\end{equation}
It is clear that given the map $h_{n+1}$ we can always choose $g_n$ such that $h_n$ is of the form \eqref{form:h_n}, and that this map $g_n$ is unique. Hence we can view $\alpha_n$ as a function of $\alpha_{n+1}$. In fact, this function is affine and given by
\begin{equation}
\alpha_n = \frac{b_n^k}{a_n} \alpha_{n+1} + \frac{c_n}{a_n},
\end{equation}
or equivalently
\begin{equation}
\alpha_{n+1} = \frac{a_n}{b_n^k} \alpha_n + \frac{c_n}{b_n^k}.
\end{equation}
Note that both coefficients of these affine maps are uniformly bounded from above in norm, and that $|\frac{a_n}{b_n^k}| \ge \frac{1}{D}$. It follows that there exist a unique bounded orbit for this sequence of affine maps. In other words, we can choose a sequence $(h_n)$ whose $k$-th degree terms are uniformly bounded. It follows directly that the maps $g_n$ are also uniformly bounded, which completes the proof.

Moreover, we note that we have much more flexibility if we have a sequence of intervals $(I_j)$ (with $j$ odd) and we want to find sequences $(h_n)$ and $(g_n)$ on each $I_j$ that are uniformly bounded over all $j$. In fact for each interval $I_j$ we can start with any value for $\alpha_{r_j - 1}$ and inductively define the maps $h_n$ backwards. We merely need to choose uniformly bounded starting values $\alpha_{r_j - 1}$.

The entire story works just the same for the intervals $I_j$ with $j$ even. In this case we have that $|a_n| \le |b_n|$, so we work with maps $h_n$ and $g_n$ of the form
\begin{equation}
h_n: (z,w) \mapsto (z , w + \beta_n z^k),
\end{equation}
and
\begin{equation}
g_n: (z,w) \mapsto (a_n z + \delta_n w^k, b_n w).
\end{equation}
The only matter to which we should pay attention is what happens where the different intervals are connected. Hence let us look at the following part of the commutative diagram
\begin{equation}
\begin{CD}
\cdots @>f_{r_j - 2}>> \mathbb{C}^2 @>f_{r_j - 1}>> \mathbb{C}^2 @>f_{r_j}>> \mathbb{C}^2 @>f_{r_j + 1}>> \cdots\\
@VV{}V @VVh_{r_j - 1}V @VVh_{r_j}V @VVh_{r_j + 1}V\\
\cdots @>g_{r_j - 2}>> \mathbb{C}^2 @>g_{r_j - 1}>> \mathbb{C}^2 @>g_{r_j}>> \mathbb{C}^2 @>g_{r_j + 1}>> \cdots
\end{CD}
\end{equation}
We consider the case where $j$ is odd and $j+1$ even; the other is similar. Imagine that we have constructed the maps $h_n$ and $g_n$ on the interval $I_{j+1}$. Then $h_{r_j}$ is of the form
\begin{equation}
h_{r_j}: (z,w) \mapsto (z , w + \beta_{r_j} z^k).
\end{equation}
As before we can find $g_{r_j -1}$ such that the map $h_{r_j - 1}$ given by $g_{r_j -1}^{-1} \circ h_{r_j} \circ f_{r_j -1}$ is of the form
\begin{equation}
h_{r_j -1}: (z,w) \mapsto (z + \alpha_{r_j -1}w^k, w).
\end{equation}
We note that while the map $g_{r_j -1}$ does depend on the coefficient $\beta_{r_j}$, the coefficient $\alpha_{r_j -1}$ (and equivalently also the map $h_{r_j -1}$) only depends on the coefficients of $f_{r_j-1}$ and not on those of $h_{r_j}$. Hence we can find uniformly bounded maps $h_n$ on each of the intervals, and therefore also uniformly bounded maps $g_n$, which are lower triangular for $n$ in each interval $[p_j, r_j)$ with $j$ odd, and upper triangular when $j$ is even. This completes the proof of Theorem \ref{thm:main}.

\section{Proof in the general case}

In what follows we drop our assumption of diagonality of the linear parts of the maps $f_n$. Our approach will be different this time. We will not able to use linear maps $l_n$ as we did in the diagonal setting. The problem is that trying to keep bounded off-diagonal terms in the linear parts of the maps $\tilde{f}_n$ places strong restrictions on the maps $l_n$, even if we use lower-triangular matrices instead of diagonal matrices.

Our solution will be to change coordinates to lower triangular maps $g_n$ immediately, but this comes at severe cost. The coefficients of the maps $g_n$ and $h_n$ will grow exponentially with $n$. What saves the day is that we can find sequences $(g_n)$ and $(h_n)$ for which the coefficients are bounded \emph{at the start and end of each train}. This gives us estimates on the coefficients of all $(g_n)$ and $(h_n)$, which turn out to be sufficient under the additional assumption that
$$
D^{11/5} < C.
$$

Our plan for proving Theorem \ref{thm:general} is the following: We use the trains defined in Definition \ref{def:trains}, where we pick $k$ such that $k < \frac{11}{5}$ and $D^k <C$, and set $x = \frac{3}{2}$. Then we will analyze the conditions on the coefficients of the maps $(g_n)$ and $(h_n)$ that arise from the commutative diagram. We prove that it is possible to chose a specific sequence of maps for which the coefficients remain bounded at the start and end of each train. In the rest of the paper we assume that we have made such a choice of maps.

In the second step we find an estimate for the growth of the coefficients of $(g_n)$ and $(h_n)$ in the middle of trains, and show that the basin of the sequence $(g_n)$ is equal to $\mathbb C^2$. In the third step we use the usual construction of the maps $\Phi_n$. Most work goes into proving that these maps converge on compact subsets to a holomorphic map $\Phi$ from the basin of the sequence $(f_n)$ to the basin of the sequence $\Omega_{(g_n)}$. We end by showing that $\Phi$ is the required biholomorphism.

The choice of $x$ we made seems rather random, but it turns out
that the exact value of $x$ does not influence our result. The
bottleneck of our estimates on the quadratic coefficients of
$(g_n)$ and $(h_n)$, depends on estimates on the wagons and the
length of the trains. When $x$ increases, the estimates on the
wagons (Lemma \ref{lemma:trein}) get weaker while the trains get
longer. These effects cancel out.
\medskip

\noindent {\bf Step 1:} Exploring new coordinates

\medskip

Instead of first conjugating with linear maps $(l_n)$ as we did in the diagonal case, we will immediately conjugate (on each train $j$) with maps $h_n^j$ (for $n \in [p_j,p_{j+1}]$) of the form
$$
h_n^j(z,w) =(z,w) + ({^j}\alpha_n^{0,2} w^2 + {^j}\alpha_n^{1,1} zw + {^j}\alpha_n^{2,0} z^2, {^j}\beta_n^{0,2} w^2 + {^j}\beta_n^{1,1} zw + {^j}\beta_n^{2,0} z^2 )
$$
We will obtain a diagram of the form
$$
\begin{CD}
\cdots @> f_{p_j-2} >> \mathbb{C}^2 @> f_{p_j-1} >> \mathbb{C}^2 @> M_j >> \mathbb{C}^2 @> f_{p_j} >> \mathbb{C}^2 @> f_{p_j + 1} >> \mathbb C^2 @> f_{p_j + 2} >> \cdots\\
@VV V @VV h_{p_j-1}^{j-1} V @VV h_{p_j}^{j-1} V @VV h_{p_j}^j V @VV h_{p_j+1}^j V @VV h_{p_j+2}^j V\\
\cdots @> g_{p_j-2} >> \mathbb{C}^2 @> g_{p_j-1} >> \mathbb{C}^2 @> M_j >> \mathbb{C}^2 @> g_{p_j} >> \mathbb{C}^2 @> g_{p_j + 1} >> \mathbb C^2 @> g_{p_j + 2} >> \cdots\\
\end{CD}
$$
which will commute up to degree two. We write
$$
f_n(z,w) = (a_n z, b_n w + c_n z) + (\sum_{m=2}^{\infty} \sum_{i=0}^{m}d_n^{i,m-i} z^iw^{m-i},\sum_{m=2}^{\infty} \sum_{i=0}^{m}e_n^{i,m-i} z^iw^{m-i})
$$
as usual. The maps $g_n$ will be lower triangular of the form
$$
g_n(z,w) = (a_n z, b_n w + c_n z + d_n z^2).
$$
Fix $j$ for now, and drop the indices $j$ in $h_n^j$.

Note that $h_n(z) = g_n^{-1} \circ h_{n+1} \circ f_n (z)+
O(||z||^3)$. This gives us the following relations between the
quadratic coefficients of $h_n$ and $h_{n+1}$:
$$
\begin{aligned}
\alpha_n^{0,2} & = \frac{b_n^2}{a_n} \cdot \alpha_{n+1}^{0,2} + x_n^{0,2},\\
\alpha_n^{1,1} & = b_n \cdot \alpha_{n+1}^{1,1} + r_n^{1,1}(\alpha_{n+1}^{0,2}) + x_n^{1,1},\\
\alpha_n^{2,0} & = a_n \cdot \alpha_{n+1}^{2,0} + r_n^{2,0}(\alpha_{n+1}^{0,2}, \alpha_{n+1}^{1,1}) + x_n^{2,0},\\
\beta_n^{0,2} & = b_n \cdot \beta_{n+1}^{0,2} + s_n^{0,2}(\alpha_{n+1}^{0,2}) + y_n^{0,2},\\
\beta_n^{1,1} & = a_n \cdot \beta_{n+1}^{1,1} + s_n^{1,1}(\alpha_{n+1}^{0,2}, \alpha_{n+1}^{1,1}, \beta_{n+1}^{0,2}) + y_n^{1,1},\\
\beta_n^{2,0} & = \frac{a_n^2}{b_n} \cdot \beta_{n+1}^{2,0} + s_n^{2,0}(\alpha_{n+1}^{0,2}, \alpha_{n+1}^{1,1}, \alpha_{n+1}^{2,0}, \beta_{n+1}^{0,2}, \beta_{n+1}^{1,1}) + y_n^{2,0}-\frac{d_n}{b_n}.
\end{aligned}
$$
Here the $r_n$'s and $s_n$'s are linear functions with
coefficients bounded by $4\frac{D^2}{C^2}$, and the constants $x_n$ and
$y_n$ have the same bound. We will usually drop the variables in the functions $r_n$ and $s_n$.

Given $h_{p_{j+1}}$, we can now set
\begin{equation}\label{eq:d}
d_n=a_n^2 \cdot \beta_{n+1}^{2,0} + b_n s_n^{2,0}(\alpha_{n+1}^{0,2}, \alpha_{n+1}^{1,1}, \alpha_{n+1}^{2,0}, \beta_{n+1}^{0,2}, \beta_{n+1}^{1,1}) +b_n y_n^{2,0}
\end{equation}
for $p_j \leq n < p_{j+1}$. Then $\beta_n^{2,0}$ will be zero for these values of $n$.

The values of $\alpha_n^{0,2}$ up to $\beta_n^{1,1}$ can be studied using the relations above. We know that $|a_n|$ and $|b_n|$ are at most $D<1$. The value of $\left|\frac{b_n^2}{a_n}\right|$ is small on average, by the choice of our trains. Therefore, when we start with $h_{p_{j+1}}$ and work backwards, the quadratic constants $\alpha_n^{0,2}$ up to $\beta_n^{1,1}$ should not get too big. This is made more precise in the lemma below.

For that lemma, pick $0< \epsilon < \min\left\{\frac{11 - 5k}{4},\frac{5-2k}{3},\frac{3-k}{2}\right\}$, which is possible since we assumed that $k < 11/5$.

\begin{lemma}\label{lemma:afschatting}
For any $\delta>0$, there exists a constant $X=X(C,D,\epsilon,\delta)$ independent of $j$, such that if
$|\alpha_{p_{j+1}}^{0,2}|, \ldots, |\beta_{p_{j+1}}^{1,1}|$ are all bounded by $X$, then for all $n \in
[p_j,p_{j+1}]$ we have:
\begin{equation*}
|\alpha_{n}^{0,2}|, \ldots, |\beta_{n}^{1,1}| \leq Y X \max_{n \leq
s \leq t \leq p_{j+1}}\left\{\left|\frac{b_{t,s}^{2-\epsilon}
}{a_{t,s}}\right|\cdot |a_{s,n}|^{1-\epsilon}\right\}\cdot
\max\left\{\delta, D^{\epsilon(p_{j+1}-n)/2}\right\},
\end{equation*}
where $Y=Y(C,D,\epsilon)$ is a constant independent of both $j$ and $\delta$.
\end{lemma}

\begin{proof}
We use backwards induction on $n$ to prove the following slightly stronger statements one by one:
\begin{align*}
|\alpha_n^{0,2}|&\leq X \max_{n \leq t \leq
p_{j+1}}\left\{\left|\frac{b_{t,n}^{2-\epsilon}}{a_{t,n}}\right|\right\}
\max\left\{\delta, D^{\epsilon(p_{j+1}-n)/2}\right\},\\
|\alpha_n^{1,1}|, |\beta_n^{0,2}|&\leq Y_1 X \max_{n \leq s \leq t \leq
p_{j+1}}\left\{\left|\frac{b_{t,s}^{2-\epsilon}}{a_{t,s}}\right|\cdot
|b_{s,n}|^{1-\epsilon}\right\}\max\left\{\delta,
D^{\epsilon(p_{j+1}-n)/2}\right\},\\
|\alpha_n^{2,0}|, |\beta_n^{1,1}| &\leq Y_2 X \max_{n \leq r \leq
s \leq t \leq p_{j+1}}\left\{\left|\frac{b_{t,s}^{2-\epsilon}
}{a_{t,s}}\right| |b_{s,r}|^{1-\epsilon} |a_{r,n}|^{1-\epsilon}\right\}
\max\left\{\delta, D^{\epsilon(p_{j+1}-n)/2}\right\},
\end{align*}
where we use the constants $X=\frac{4D^2}{C^2 \delta(1-D^{\epsilon/2})}$, $Y_1=\frac{4D^2}{C^2C^{1-\epsilon}D^{\epsilon/2}(1-D^{\epsilon/2})}+1$ and $Y_2=Y_1 \cdot \left(\frac{8D^2}{C^2C^{1-\epsilon}D^{\epsilon/2}(1-D^{\epsilon/2})}+1 \right)$.

The case $n=p_{j+1}$ is
stated in the assumptions, so we can start our backwards induction. Suppose that the first inequality holds
for $n+1$. Then we find that
\begin{align*}
|\alpha_{n}^{0,2}|&=\left|\frac{b_n^2}{a_n}\alpha_{n+1}^{0,2} + x_n^{0,2}\right|\\
& \leq \left|\frac{b_n^2}{a_n}\right| \cdot X \cdot \max_{n+1 \leq t
\leq
p_{j+1}}\left\{\left|\frac{b_{t,n+1}^{2-\epsilon}}{a_{t,n+1}}\right|\right\}\cdot \max\left\{\delta, D^{\epsilon(p_{j+1}-n-1)/2}\right\}+4\frac{D^2}{C^2}.
\end{align*}
Now use that $\left|\frac{b_n^2}{a_n}\right| = |b_n|^{\epsilon} \cdot \left|\frac{b_n^{2-\epsilon}}{a_n}\right|$ and $|b_n|^{\epsilon} \le D^{\epsilon/2} \cdot D^{\epsilon/2}$ to obtain
\begin{align*}
|\alpha_{n}^{0,2}| \leq X
\max_{n+1 \leq t
\leq
p_{j+1}}\left\{\left|\frac{b_{t,n}^{2-\epsilon}}{a_{t,n}}\right|\right\}\cdot
\max\left\{\delta, D^{\epsilon(p_{j+1}-n)/2}\right\},
\end{align*}
for $X$ as stated above. The other two cases are proved similarly and are left for the reader. For the upper bound on $|\alpha_n^{1,1}|$, $|\alpha_n^{2,0}|$, $|\beta_n^{0,2}|$ and $|\beta_n^{1,1}|$ we note that
$$
\max_{n \leq r \leq s \leq t \leq
p_{j+1}}\left\{\left|\frac{b_{t,s}^{2-\epsilon} }{a_{t,s}}\right|\cdot
|b_{s,r}|^{1-\epsilon}|a_{r,n}|^{1-\epsilon}\right\}=\max_{n \leq
s \leq t \leq p_{j+1}}\left\{\left|\frac{b_{t,s}^{2-\epsilon}
}{a_{t,s}}\right|\cdot |a_{s,n}|^{1-\epsilon}\right\}.
$$
\end{proof}

We obtain the following similar estimate on the coefficients of $g$.

\begin{corollary}\label{cor:d_en_beta}
If the assumptions of Lemma \ref{lemma:afschatting} are satisfied,
and $|\beta_{p_{j+1}}^{2,0}|\leq X$, then for all $n \in [p_j,
p_{j+1})$ we have that $|\beta_n^{2,0}|=0$, and
$$
|d_n|  \leq
\frac{4D^2}{C^2}(6YX+1)\cdot
\max_{n+1 \leq s \leq t \leq
p_{j+1}}\left\{\left|\frac{b_{t,s}^{2-\epsilon} }{a_{t,s}}\right|\cdot
|a_{s,n+1}|^{1-\epsilon}\right\}.
$$
\end{corollary}
\begin{proof}
The first statement follows directly from the choice of $d_n$ in
(\ref{eq:d}). For the second statement, we have
\begin{align*}
|d_{p_{j+1}-1}|=|a_{p_{j+1}-1}^2 \cdot \beta_{{p_{j+1}}}^{2,0}+
b_{p_{j+1}-1} s_{p_{j+1}-1}^{2,0} +b_{p_{j+1}-1}
y_{p_{j+1}-1}^{2,0}|
\leq \frac{4D^2}{C^2}(6YX+1).
\end{align*}
For $n \in [p_j,p_{j+1}-1)$ we similarly have
\begin{align*}
|d_n| \leq \frac{4D^2}{C^2}(5YX+1)\max_{n+1 \leq s
\leq t \leq p_{j+1}}\left\{\left|\frac{b_{t,s}^{2-\epsilon}
}{a_{t,s}}\right|\cdot |a_{s,n+1}|^{1-\epsilon}\right\}.
\end{align*}
\end{proof}

In the following lemma we will give an upper bound for
$$
\max_{p_j \leq s
\leq t \leq p_{j+1}}\left\{\left|\frac{b_{t,s}^{2-\epsilon}
}{a_{t,s}}\right|\cdot |a_{s,n}|^{1-\epsilon}\right\},
$$
which plays an important role in the above estimates. This allows us to estimate the quadratic terms at the beginning of train $j$.

\begin{lemma}\label{lemma:maxpj+1pj}
We have
$$
\max_{p_j \leq s \leq t \leq
p_{j+1}}\left\{\left|\frac{b_{t,s}^{2-\epsilon} }{a_{t,s}}\right|\cdot
|a_{s,p_j}|^{1-\epsilon}\right\}\leq K_3,
$$
where $K_3=K_3(C,D,k,x)$ is a constant independent of $j$.
\end{lemma}
\begin{proof}
Suppose that $p_j \leq s \leq t \leq p_{j+1}$ and look at
$\left|\frac{b_{t,s}^{2-\epsilon}}{a_{t,s}}\right| \cdot
|a_{s,p_j}|^{1-\epsilon}$. We consider the cases $t \geq q_j$ and $t\leq q_j$ separately.

\medskip

\noindent {\bf Case I}: $t \geq q_j$.

\medskip

We have
\begin{align}\label{eq:CaseIi}
\left|\frac{b_{t,s}^{2-\epsilon}}{a_{t,s}}\right| \cdot
|a_{s,p_j}|^{1-\epsilon} \leq \max_{q_j \leq i \leq
t}\left\{\left|\frac{b_{t,i}^{2-\epsilon}}{a_{t,i}}\right|\right\} \cdot \max_{p_j
\leq i \leq
q_j}\left\{\left|\frac{b_{q_j,i}^{2-\epsilon}}{a_{q_j,i}}\right|\cdot
|a_{i,p_j}|^{1-\epsilon}\right\}.
\end{align}
Now notice that
\begin{equation}\label{eq:eerste}
\left|\frac{b_{t,i}^{2-\epsilon}}{a_{t,i}}\right| \leq
\frac{D^{(2-\epsilon)(t-i)}}{D^{k(t-i)}} \le D^{-
(k-2+\epsilon)(t-i)},
\end{equation}
and since $q_j \leq i \le t \le p_{j+1}$, Lemma
\ref{lemma:trein}(ii) gives us also
\begin{equation}\label{eq:tweede}
\left|\frac{b_{t,i}^{2-\epsilon}}{a_{t,i}}\right|=\left|\frac{b_{t,i}^x}{a_{t,i}}\right|
\cdot |b_{t,i}^{2-x-\epsilon}| \le  D^{(t-i)(2-x-\epsilon)} \cdot D^{-2^{j+1}}.
\end{equation}
Note that when $t-i$ increases, the first estimate will
increase while the second estimate will decrease. The estimates
are equal when $D^{-(k-x)(t-i)}=D^{-2^{j+1}}$, which is when
$t-i=\frac{2^{j+1}}{k-x}$. Using estimate \eqref{eq:eerste}
for $t-i \leq \frac{2^{j+1}}{k-x}$ and the estimate
\eqref{eq:tweede} when $t-i \geq \frac{2^{j+1}}{k-x}$, we get:
\begin{equation}\label{eq:CaseI}
\max_{q_j \leq i \leq
t}\left\{\left|\frac{b_{t,i}^{2-\epsilon}}{a_{t,i}}\right|\right\} \leq
D^{-2^{j+1}\cdot \frac{k-2+\epsilon}{k-x}}.
\end{equation}
To estimate the second term in the product in \eqref{eq:CaseIi},
use Theorem \ref{thm:trein} to see that
$\left|\frac{b_{q_j,i}}{a_{q_j,i}^x}\right|\leq K_2$, when $p_j
\leq i \leq q_j$. And by Lemma \ref{lemma:trein} we know that
$q_j-p_j \geq \frac{2^j}{k-x}$. Combining these facts gives
$$
\left|\frac{b_{q_j,i}^{2-\epsilon}}{a_{q_j,i}}\right|\cdot
|a_{i,p_j}|^{1-\epsilon} \leq \left|\frac{b_{q_j,i}}{a_{q_j,i}^x}\right| ^{1/x}\cdot
|b_{q_j,i}|^{1-\epsilon}|a_{i,p_j}|^{1-\epsilon} \leq K_x^{1/x} \cdot
D^{\frac{1-\epsilon}{k-x}\cdot 2^j}.
$$
Since $\epsilon<\frac{5-2k}{3}$ we obtain
\begin{equation}\label{eq:CaseIa}
\begin{aligned}
\left|\frac{b_{t,s}^{2-\epsilon}}{a_{t,s}}\right| \cdot
|a_{s,p_j}|^{1-\epsilon} & \leq \max_{q_j \leq i \leq
t}\left\{\left|\frac{b_{t,i}^{2-\epsilon}}{a_{t,i}}\right|\right\} \cdot \max_{p_j
\leq i \leq
q_j}\left\{\left|\frac{b_{q_j,i}^{2-\epsilon}}{a_{q_j,i}}\right|\cdot
|a_{i,p_j}|^{1-\epsilon}\right\}\\
& \leq D^{-2^{j+1}\cdot \frac{k-2+\epsilon}{k-x}} \cdot K_2^{1/x} \cdot D^{\frac{1-\epsilon}{k-x}\cdot 2^j} \leq K_2^{1/x}.
\end{aligned}
\end{equation}

\medskip

\noindent {\bf Case II}: $t \leq q_j$.

\medskip

By Theorem \ref{thm:trein} we have
$\left|\frac{b_{t,p_j}}{a_{t,p_j}^x}\right| \leq K_2 D^{-\frac{k-x}{x}(n-p_j)}$, and therefore
\begin{align*}
\left|\frac{b_{t,s}^{2-\epsilon}}{a_{t,s}}\right| \cdot
|a_{s,p_j}|^{1-\epsilon} &=\left| \frac{b_{t,p_j}}{a_{t,p_j}^x}\right|^{\frac{1}{2}}\cdot |a_{t,p_j}|^{\frac{x-1}{2}} \cdot \left|\frac{b_{t,s}^{\frac{3}{2}-\epsilon}}{a_{t,s}^{\frac{1}{2}}}\right| \cdot \left|\frac{a_{s,p_j}^{\frac{3}{2}-\epsilon}}{b_{s,p_j}^{\frac{1}{2}}}\right|\\
&\leq K_2^{\frac{1}{2}} \cdot D^{-\frac{k-x}{2x}(n-p_j)} \cdot D^{\frac{x-1}{2}(n-p_j)} \cdot \frac{D^{(\frac{3}{2}-\epsilon)(n-p_j)}}{D^{(k/2)(n-p_j)}}\\
& \leq K_2^{\frac{1}{2}},
\end{align*}
where we used that $k<x^2$ and $\epsilon< \frac{3-k}{2}$. Setting $K_3=K_2^{1/x}$ now gives us the desired result.
\end{proof}

Suppose we are now given $h_{p_j}^j$ and set $h_{p_j}^{j-1} =
M_j^{-1} \circ h_{p_j}^j \circ M_j$. Recall that the $M_j$ are
unitary matrices arising from the definition of the trains. A
consequence of the fact that each map has at most $6$ quadratic
terms is the following.

\begin{lemma}\label{lemma:overgang}
If the coefficients of $h_{p_j}^j$ are bounded by $R>0$, then the coefficients of $h_{p_j}^{j-1}$ are bounded by $6R$.
\end{lemma}
\begin{proof}
We know that $\|h_{p_j}^j(z,w)-(z,w)\|\leq R\cdot
\|(z,w)\|^2$. Since $M_j$ is unitary we obtain
\begin{align*}
\|h_{p_j}^{j-1}(z,w)-(z,w)\|&=\|\left(h_{p_j}^j-\textrm{id}\right)(M_j(z,w))\|\\
&\leq 6R \cdot \|M_j(z,w)\|^2\\
&= 6R \cdot \|(z,w)\|^2
\end{align*}
The lemma follows.\end{proof}

The following is now an immediate consequence of Lemmas \ref{lemma:afschatting}, \ref{lemma:maxpj+1pj} and \ref{lemma:overgang}.

\begin{corollary}\label{cor:pj+1pj}
If $\delta>0$ is chosen sufficiently small, and $j_0$ is
sufficiently large, then for all $j \ge j_0$ we have that if the
coefficients of $h_{p_{j+1}}^j$ are bounded by $X$, then the
coefficients of $h_{p_j}^{j-1}$ are bounded by $X$.
\end{corollary}

With this lemma, we now have all ingredients to actually define
the sequence $(h_n)$. Let $S\subset \mathbb C^6$ be the compact
set of all second degree polynomials of the form $h = Id + O(2)$
with second degree coefficients bounded by $X$. By Lemmas
\ref{lemma:afschatting}, \ref{lemma:overgang} and Corollary
\ref{cor:pj+1pj}, a choice of $h_j = h_{p_{j+1}}^{j} \in S$
uniquely determines a map $\psi_{j-1}(h_j) := h_{p_{j}}^{j-1}$.
For every $i$ we define
$$
V_i = \bigcap_{j\ge i} \psi_i \circ \ldots \psi_j(S).
$$
Each $V_i$ is a decreasing intersection of non-empty compact sets, and hence non-empty and compact. By continuity of each $\psi_i$ it follows that
$$
V_i = \psi_i(V_{i-1}).
$$
Therefore there exist an inverse orbit of the sequence $\psi_0, \psi_1, \ldots$. That is, a sequence $h_0, h_1, \ldots$ in $S$ with $\psi_i(h_{i+1}) = h_i$. Thus we have proved the following.

\begin{thm}
We can find sequences $(h_n)$ and $(g_n)$ whose coefficients are bounded by $X$ at the start and end of every train $j$, where $j \geq j_0$.
\end{thm}

\medskip

\noindent{\bf Step 2:} Properties of these new coordinates.

\medskip

What remains to show is that with the construction from the previous step gives a basin $\Omega_{(g_n)}$ which is on the one hand equal to all of $\mathbb C^2$, and on the other hand equivalent to the basin $\Omega_{(f_n)}$. In order to draw these conclusions we will need finer estimates on the size of the coefficients of the maps $h_n$ and $g_n$.

\begin{lemma}\label{lemma:maxnpj+1}
For $n \in [p_j,p_{j+1})$ we have
$$
\max_{n \leq s \leq t \leq
p_{j+1}}\left\{\left|\frac{b_{t,s}^{2-\epsilon} }{a_{t,s}}\right|\cdot
|a_{s,n}|^{1-\epsilon}\right\} \leq K_3 \cdot
D^{-2n(k-2+\epsilon)},
$$
where $K_3=K_3(C,D,k,x)$ as in Lemma \ref{lemma:maxpj+1pj}.
\end{lemma}

\begin{proof}
First, note that
$$\max_{n \leq s \leq t \leq
p_{j+1}}\left\{\left|\frac{b_{t,s}^{2-\epsilon} }{a_{t,s}}\right|\cdot
|a_{s,n}|^{1-\epsilon}\right\} \leq \max_{n \leq s \leq t \leq
p_{j+1}}\left\{\left|\frac{b_{t,s}^{2-\epsilon} }{a_{t,s}}\right|\right\}.
$$
Now suppose $p_j \leq n \leq s \leq t \leq p_{j+1}$. As in the
proof of Lemma \ref{lemma:maxpj+1pj}, we can distinguish several
cases:
\begin{enumerate}
\item[(I)]{$q_j \leq s \leq t \leq p_{j+1}$}
\item[(II)]{$p_j\leq s \leq q_j \leq t \leq p_{j+1}$}
\item[(III)]{$p_j \leq s \leq t \leq q_j$}
\end{enumerate}

In case (I), the proof of Lemma \ref{lemma:maxpj+1pj} gives us
equation \ref{eq:CaseI}, which implies that
\begin{equation}\label{eq:CaseIgeneral}
\left|\frac{b_{t,s}^{2-\epsilon}}{a_{t,s}}\right| \leq D^{-2^{j+1}\cdot
\frac{k-2+\epsilon}{k-3/2}}.
\end{equation}

\medskip

For case (II), we can use equation \ref{eq:CaseI} again to see
that
$$
\left|\frac{b_{t,q_j}^{2-\epsilon}}{a_{t,q_j}}\right| \leq
D^{-2^{j+1}\cdot \frac{k-2+\epsilon}{k-3/2}}.
$$
Next, Theorem \ref{thm:trein} shows that
$$
\left|\frac{b_{q_j,s}^{2-\epsilon}}{a_{q_j,s}}\right| \leq \left|\frac{b_{q_j,s}^{1/x}}{a_{q_j,s}}\right|
\leq K_2^{1/x}=K_3.
$$
Combining these estimates gives
\begin{equation}\label{eq:CaseIIgeneral}
\left|\frac{b_{t,s}^{2-\epsilon}}{a_{t,s}}\right|\leq K_3 \cdot
D^{-2^{j+1}\cdot \frac{k-2+\epsilon}{k-3/2}}.
\end{equation}

\medskip

For the remaining case (III) we can once again apply Theorem \ref{thm:trein}:
\begin{equation}\label{eq:eersteIII}
\left|\frac{b_{t,s}^{2-\epsilon}}{a_{t,s}}\right|=\left|\frac{b_{t,s}}{a_{t,s}^x} \right|^{1/x} \cdot |b_{t,s}|^{2-\frac{1}{x}-\epsilon}
\leq K_2^{1/x} D^{-2^j\frac{x+1}{x^2}} D^{(t-s)(2-\epsilon-\frac{1}{x})}.
\end{equation}
On the other hand, we have the easy estimate
\begin{equation}\label{eq:tweedeIII}
\left|\frac{b_{t,s}^{2-\epsilon}}{a_{t,s}}\right| \leq
\frac{D^{(t-s)(2-\epsilon)}}{C^{t-s}} \leq
\frac{D^{(t-s)(2-\epsilon)}}{D^{k(t-s)}}=D^{-(k-2+\epsilon)(t-s)}.
\end{equation}
Note that as $t-s$ increases, estimate \ref{eq:eersteIII} will
decrease while estimate \ref{eq:tweedeIII} will increase. When
$t-s \leq \frac{2^{j+1}}{k-x}$, we can therefore use estimate
\ref{eq:tweedeIII} to get
\begin{equation*}
\left|\frac{b_{t,s}^{2-\epsilon}}{a_{t,s}}\right| \leq
D^{-(k-2+\epsilon)(t-s)}\leq D^{-\frac{k-2+\epsilon}{k-x}2^{j+1}}.
\end{equation*}
And when $t-s \geq \frac{2^{j+1}}{k-x}$, estimate \ref{eq:eersteIII}
shows that
\begin{align*}
\left|\frac{b_{t,s}^{2-\epsilon}}{a_{t,s}}\right| &\leq K_2^{1/x} D^{-2^j\frac{x+1}{x^2}} D^{(t-s)(2-\epsilon-\frac{1}{x})}
&& \leq K_2^{1/x} D^{-2^j\frac{x+1}{x^2}} D^{\frac{2-\epsilon-1/x}{k-x}2^{j+1}}\\
& = K_2^{1/x} \left(D^{-2^{j+1}} \right)^{\frac{x+1}{2x^2}-\frac{2-\epsilon-1/x}{k-x}}
&& = K_2^{1/x} \left(D^{-2^{j+1}} \right)^{\frac{\frac{5k}{9}-\frac{13}{6}+\epsilon}{k-x}}\\
& \leq K_2^{1/x} &&=K_3,
\end{align*}
since $k<\frac{11}{5}$ and $x =\frac{3}{2}$.
So we can conclude Case (III) with
\begin{equation}\label{eq:CaseIIIgeneral}
\left|\frac{b_{t,s}^{2-\epsilon}}{a_{t,s}}\right|\leq K_3
D^{-\frac{k-2+\epsilon}{k-x}2^{j+1}}.
\end{equation}

Finally note that $n\geq p_j\geq q_0+\sum_{i=1}^{j-1}(q_i-p_i)\geq
\frac{2}{k-x}+\sum_{i=1}^{j-1} \frac{2^i}{k-x}=\frac{2^j}{k-x}$ by
Lemma \ref{lemma:trein}. Combining this with equations \ref{eq:CaseIgeneral}, \ref{eq:CaseIIgeneral} and \ref{eq:CaseIIIgeneral} proves that
\begin{equation*}
\left|\frac{b_{t,s}^{2-\epsilon}}{a_{t,s}}\right| \leq K_3 D^{-2^{j+1}\cdot
\frac{k-2+\epsilon}{k-x}} \leq K_3 D^{-2n(k-2+\epsilon)}.
\end{equation*}
\end{proof}

Lemmas \ref{lemma:maxnpj+1} and \ref{lemma:afschatting} now give us an easy estimate on the coefficients of our maps $h^j_n$ and $g_n$.

\begin{corollary}
For $n \in [p_j,p_{j+1}]$ and $j \geq j_0$, all quadratic coefficients of the maps $g_n$ and $h^j_n$ are bounded by
$$
\frac{4D^2K_3}{C^2}(6YX+1) \cdot D^{-2n(k-2+\epsilon)}.
$$
Since there are only finitely many functions for $j < j_0$, we can now pick a large constant $Z$ such that the quadratic coefficients of all maps $g_n$ and $h^j_n$ are bounded by
$$
Z \cdot D^{-2n(k-2+\epsilon)}.
$$
\end{corollary}

Using these estimates, we can find $\Omega_{(g_n)}$:

\begin{lemma}\label{lemma:everything}
The basin of the sequence $g_1, \ldots g_{p_1-1}, M_1, g_{p_1}, \ldots$ as constructed is equal to $\mathbb C^2$.
\end{lemma}

\begin{proof}
Let $G_j=M_j \circ g_{p_{j+1},p_j}$ and write $G_{j,i} = G_{j-1}\circ \ldots \circ G_i$ as usual. For $n \in [p_j,p_{j+1})$, we have $g_n(z)=L_n(z)+\left(0,d_nz_1^2\right)$, where $L_n$ is the linear part of $f$. Hence we have
$$
g_{p_{j+1},p_j}(z)=L_{p_{j+1},p_j}(z)+\left(0, \sum_{i=p_j}^{p_{j+1}-1} b_{p_{j+1},i+1}d_ia_{i,p_j}^2z_1^2 \right).
$$
By Corollary \ref{cor:d_en_beta}, we know that
$$
|d_i| \leq Z \max_{i+1 \leq s \leq t \leq p_{j+1}} \left\{\left|\frac{b_{t,s}^{2-\epsilon}}{a_{t,s}}\right||a_{s,n}|^{1-\epsilon}\right\},
$$
for $i \in [p_{j},p_{j+1})$. Therefore we find
\begin{align*}
\left|\sum_{i=p_j}^{p_{j+1}-1} b_{p_{j+1},i+1}d_ia_{i,p_j}^2\right| \leq & \sum_{i=p_j}^{p_{j+1}-1} Z|b_{p_{j+1},i+1}|^{1/2}|a_{i,p_j}|^2 \max_{i+1 \leq s \leq t \leq p_{j+1}} \left\{\left|\frac{b_{t,s}^{\frac{5}{2}-\epsilon}}{a_{t,s}}\right||a_{s,n}|^{1-\epsilon}\right\}\\
\leq & \sum_{i=p_j}^{p_{j+1}-1} Z|b_{p_{j+1},i+1}|^{1/2}|a_{i,p_j}|^2 \\
\leq & \sum_{i=p_j}^{p_{j+1}-1} \frac{Z}{C}D^{\frac{p_{j+1}-p_j}{2}}D^{i-p_j}\\
\leq & \frac{Z}{C(1-D)}D^{\frac{2^{j-1}}{k-x}},
\end{align*}
where the last step follows by Lemma \ref{lemma:trein}(iii). Given $z \in \mathbb C^2$ we write $R_j = \|G_{j,0}(z)\|$.
Since $M_j$ is unitary, we obtain
$$
R_{j+1} \leq D^{\frac{2^j}{k-x}} R_j +\frac{Z}{C(1-D)} D^{\frac{2^{j-1}}{k-x}}R_j^2.
$$
Hence we have the following two cases.
\begin{enumerate}
\item[(a)]{If $R_j \leq D^{\frac{2^{j-1}}{k-x}}$:\;\;\; $R_{j+1} \leq \left(1+\frac{Z}{C(1-D)} \right)D^{\frac{2^j}{k-x}} R_j$.}
\item[(b)]{If $R_j > D^{\frac{2^{j-1}}{k-x}}$:\;\;\; $R_{j+1} \leq \left(1+\frac{Z}{C(1-D)} \right)D^{\frac{2^{j-1}}{k-x}} R_{j+1}^2$.}
\end{enumerate}
Let $j_1$ be such that $\left(1+\frac{Z}{C(1-D)} \right)D^{\frac{2^{j_1-1}}{k-x}} < \frac{1}{2}$. Then we can derive from (a) that for $j \geq j_1$ and $R_j \leq D^{\frac{2^{j-1}}{k-x}}$ we have
$$
R_{j+1} \leq \frac{1}{2}D^{\frac{2^j}{k-x}}.
$$
Hence if $R_j \le D^{\frac{2^{j-1}}{k-x}}$ for some $j \ge j_1$, then the same holds for all larger $j$.

Suppose for the purpose of a contradiction that $R_j > D^{\frac{2^{j-1}}{k-x}}$ for each $j \geq j_1$. Then for $j \geq j_1$ we have:
\begin{align*}
R_{j+1} & \leq \left(1+ \frac{Z}{C(1-D)} \right)D^{\frac{2^{j-1}}{k-x}} R_j^2\\
& \leq \left(\frac{1}{2} \right)^{2^{j-j_1}} R_j^2.
\end{align*}
Hence for all $j \ge j_1$ sufficiently large we have
\begin{align*}
R_j = \|G_{j,j_1}\left( G_{j_1,0}(z) \right)\| \leq \left( \frac{1}{2}\right)^{(j-j_1)2^{j-1-j_1}}(R_{j_1})^{2^{j-j_1}} \leq D^{\frac{2^{j-1}}{k-x}}.
\end{align*}
This contradicts the assumption that $R_j>D^{\frac{2^{j-1}}{k-x}}$ for each $j \geq j_1$. Hence there exists $j_z \geq j_1$ such that $R_j \leq D^{\frac{2^{j-1}}{k-x}}$ for all $j \ge j_z$.

Now let $n \in [p_j,p_{j+1})$ and $j \geq j_z$, and recall that
$$
g_{n,p_j}(z)=L_{n,p_j}(z)+\left(0, \sum_{i=p_j}^{n-1} b_{n,i+1}d_ia_{i,p_j}^2z_1^2 \right).
$$
Using Lemma \ref{lemma:maxnpj+1} and its proof we find
$$
|b_{n,i+1}d_ia_{i,p_j}^2| \leq Z \cdot D^{-2^{j+1}\frac{k-2+\epsilon}{k-x}} |a_{i,p_j}|,
$$
and hence
$$
|\sum_{i=p_j}^{n-1} b_{n,i+1}d_ia_{i,p_j}^2| \leq \frac{Z}{1-D}D^{-2^{j+1}\frac{k-2+\epsilon}{k-x}}.
$$
For $j \geq j_z$ and $n \in [p_j,p_{j+1})$ we therefore find that
\begin{align*}
\|g_{n,1}(z)\|&=\|g_{n,p_j}\left(G_j(z)\right)\|\\
 & \leq D^{n-p_j}D^{\frac{2^{j-1}}{k-x}}+ \frac{Z}{1-D}D^{-2^{j+1}\frac{k-2+\epsilon}{k-x}}\left(D^{\frac{2^{j-1}}{k-x}}\right)^2\\
& \leq D^{\frac{2^{j-1}}{k-x}} + \frac{Z}{1-D}D^{2^{j+1}\frac{-k+5/2-\epsilon}{k-x}}.
\end{align*}
Since $\epsilon < \frac{5-2k}{2}$, we see that $\|g_{n,1}(z)\| \rightarrow 0$ as $n \rightarrow \infty$, which completes the proof.
\end{proof}

\medskip

\noindent {\bf Step 3:} The biholomorphism.

\medskip

What is left to show is that the basins of the sequences
$$
f_0, \ldots f_{p_1-1}, M_1, f_{p_1}, \ldots f_{p_2-1}, M_2, f_{p_2}, \ldots
$$ and
$$
g_0, \ldots g_{p_1-1}, M_1, g_{p_1}, \ldots g_{p_2-1}, M_2, g_{p_2}, \ldots
$$ are equivalent. We can simplify notation by composing the matrices $M_j$ with neighbouring functions. Define:
\begin{align}
\bar{f}_n&=\begin{cases} f_{p_j} \circ M_j &
\textrm{ if }n=p_{j}\\
f_n & \textrm{ if } n \notin \{p_1,p_2,\ldots\}
\end{cases}\\
\bar{g}_n&=\begin{cases} g_{p_j} \circ M_j &
\textrm{ if }n=p_{j}\\
g_n & \textrm{ if } n \notin \{p_1,p_2,\ldots\}
\end{cases}\\
\bar{h}_n&=h_n^j  \textrm{ if } n \in [p_j+1,p_{j+1}]&
\end{align}
Hence from now on we need to study the basins of the sequences $(\bar{f}_n)$ and $(\bar{g}_n)$. We note that the sequence $(\bar{f_n})$ still satisfies
$$
C\|z\| \leq \|\bar{f}_n(z)\|\leq D\|z\|
$$
for $z \in \mathbb B$. The following diagram commutes up to degree 2:
$$
\begin{CD}
 \mathbb{C}^2 @> \bar{f}_0 >> \mathbb{C}^2 @> \bar{f}_1 >> \mathbb{C}^2 @> \bar{f}_2 >> \mathbb{C}^2 @> \bar{f}_3 >> \mathbb{C}^2 @> \bar{f}_4 >> \mathbb C^2 @> \bar{f}_5 >> \cdots\\
@VV \bar{h}_0 V @VV \bar{h}_1 V @VV \bar{h}_2 V @VV \bar{h}_3 V @VV \bar{h}_4 V @VV \bar{h}_5 V\\
\mathbb{C}^2 @> \bar{g}_0 >> \mathbb{C}^2 @> \bar{g}_1 >> \mathbb{C}^2 @> \bar{g}_2 >> \mathbb{C}^2 @> \bar{g}_3 >> \mathbb{C}^2 @> \bar{g}_4 >> \mathbb C^2 @> \bar{g}_5 >> \cdots\\
\end{CD}
$$
As usual we define
$$
\Phi_n := \bar{g}_{n,0}^{-1} \circ \bar{h}_n \circ \bar{f}_{n,0}.
$$
It will also be convenient to write
$$
\phi_{n,m} = \bar{g}_{n,m}^{-1} \circ \bar{h}_n \circ \bar{f}_{n,m}.
$$
Our goal is to show that the sequence $(\Phi_n)$ converges, uniformly on compact subsets of the basin $\Omega_{(\bar{f}_n)}$ to an
isomorphism from $\Omega_{(\bar{f}_n)}$ to $\Omega_{(\bar{g}_n)} = \mathbb C^2$. In order to
accomplish this we need a number of estimates.

Pick a constant $\lambda<1$ such that $D^k < \lambda C$, and define the radius
$$
r_n=\min\left\{ \frac{(1-\lambda)C^2}{2Z}D^{2n(k-2+\epsilon)},
\frac{1}{48Z} \cdot
D^{2n(k-2+\epsilon)},\frac{1}{2}D^{\frac{4n(k-2+\epsilon)}{3-k}}\right\}.
$$
All estimates for $\bar{f}_n$, $\bar{g}_n$ and $\bar{h}_n$ will be valid on $\mathbb{B}(0,r_n)$.

\begin{lemma}\label{lemma:part2}
For $n \in \mathbb{N}$ and $z,w \in \mathbb{B}(0,r_n)$, we have
$$
\|\bar{g}_n^{-1}(z)-\bar{g}_n^{-1}(w)\|\leq \frac{1}{\lambda C} \|z-w\|.
$$
\end{lemma}
\begin{proof}
Writing $g_n(z_1,z_2)=L_n(z_1,z_2)+(0,d_n
z_1^2)$, we obtain
\begin{equation*}
g_n^{-1}(z_1,z_2)=L_n^{-1}(z_1,z_2)+\left(0,
\frac{-d_n}{a_n^2b_n}z_1^2 \right).
\end{equation*}
And hence for two points $z,w \in \mathbb{C}^2$ we get
\begin{align*}
\|g_n^{-1}(z)-g_n^{-1}(w)\| &=\|
L_n^{-1}(z-w)+\frac{-d_n}{a_n^2b_n}\left(0,z_1^2-w_1^2\right)\|\\
& \leq \frac{1}{C} \|z-w\|+\frac{|d_n|}{C^3}\|z-w\|\cdot \|z+w\|\\
&\leq \frac{1}{C} \|z-w\| \cdot \left(1+ \frac{Z
D^{-2n(k-2+\epsilon)}}{C^2}\|z+w\|\right)
\end{align*}
Using our assumptions on $\lambda<1$, $\|z\|$ and $\|w\|$ we have
\begin{align*}
\|g_n^{-1}(z)-g_n^{-1}(w)\|
 \leq \frac{1}{\lambda C} \|z-w\| \cdot \left(\lambda+ \lambda
(1-\lambda)\right)
 \leq \frac{1}{\lambda C} \|z-w\|.
\end{align*}
Since the matrices $M_j$ are unitary we immediate obtain the same estimates for the maps $\bar{g}_n$.
\end{proof}

\begin{lemma}\label{lemma:part3}
Let $n \in \mathbb N$. If $z \in \mathbb{B}(0,r_n)$, we have
$$
\|\bar{h}_n(z)\|\leq 2 \|z\|.
$$
\end{lemma}

The proof follows immediately from our estimates on the coefficients of the maps $h_n$.

\begin{lemma}\label{lemma:part4}
There exists a constant $M>0$ such that for all $n \in \mathbb N$ and $z \in \mathbb{B}(0,r_n)$,
we have
$$
\|\bar{h}_n(z)-\bar{g}_n^{-1} \circ \bar{h}_{n+1} \circ \bar{f}_n (z)\| \leq M \cdot \|z\|^k.
$$
\end{lemma}
\begin{proof}
Recall that the linear and quadratic parts of the map
$\bar{g}_n^{-1} \circ \bar{h}_{n+1} \circ \bar{f}_n$ are exactly
equal to $\bar{h}_n$. Therefore we need to find estimates on the
higher order terms of $\bar{g}_n^{-1} \circ \bar{h}_{n+1} \circ
\bar{f}_n$. Since the matrices $M_j$ are unitary, we can work with
the $f$'s, $g$'s and $h$'s instead.

Since the maps $f_n$ are uniformly attracting, it follows from the Cauchy-estimates that the degree $m$ terms of $f_n$ are bounded by $D (\sqrt{2})^m$.

A tedious but straightforward computation shows that the third order part is bounded by
\begin{align*}
\frac{19D^2Z^2}{C^3}D^{-4n(k-2+\epsilon)} 3^4
\sqrt{2}^3 \|(z_1,z_2)\|^3,
\end{align*}
and the part of degree $\ell\ge 4$ is bounded by
\begin{align*}
\frac{28D^2 Z^3}{C^3}D^{-6n(k-2+\epsilon)} \ell^5
\sqrt{2}^\ell \|(z_1,z_2)\|^\ell.
\end{align*}
Plugging in our bound on $\|z\|$ we obtain the required estimate.
\end{proof}

Our next goal is to combine the estimates above to prove that
$\varphi_{n,m}$ is close to $\bar{h}_m$ on some neighborhood of
the origin. To achieve this, we need to make sure that we can
apply our estimates every step of the way. Therefore, we define a new
constant
$$s_n=\left(\frac{M}{1-\frac{D^k}{\lambda C}}+2 \right)^{-1}r_n,$$
to give ourselves some wiggle room.

Since $k<\frac{11}{5}$ and $\epsilon<\frac{11-5k}{4}$, we have
$\frac{4(k-2+\epsilon)}{3-k} < 1$. Therefore $\frac{s_n}{D^n}$
increases exponentially and $Ds_n < s_{n+1}$. So if $z \in
\mathbb{B}(0,s_n)$, then $\bar{f}_n(z) \in \mathbb{B}(0,s_{n+1})$.
This is the point where we've really used that $k < \frac{11}{5}$.

Write $V_u=\bar{f}_{u,1}^{-1}\left(\mathbb{B}(0,1)\right)$ for $u
\in \mathbb{N}$. Next, let $v(u) \in \mathbb N$ be minimal such
that
\begin{equation}\label{eq:v(u)}
D^{v(u)-u} <s_{v(u)}.
\end{equation}
For any $m \geq v(u)$ and $z \in V_u$, we now have $\|\bar{f}_{m,1}(z)\|=\|\bar{f}_{m,u}\left(\bar{f}_{u,1}(z)\right)\| \leq D^{m-u} \leq s_m$.

\begin{lemma}\label{lemma:afstanden}
For all $m , n \in \mathbb{N}$ and $w \in \mathbb{B}(0,s_m)$ we have
\begin{enumerate}
\item[(i)]{
$\|\varphi_{n+m,m} \left(w\right)-\bar{h}_{m}(w)\|\leq
\left(\frac{M}{1-\frac{D^k}{\lambda C}} \right)
\|w\|^k,$}
\item[(ii)]{$\|\varphi_{n+m, m}
(w)\| \leq
\left(\frac{M}{1-\frac{D^k}{\lambda C}}+2 \right)\|w\|.$}
\end{enumerate}
\end{lemma}

\begin{proof}
We prove both statements simultaneously by induction on
$n$. For $n=0$ the first statement is trivial, and the second statement follows immediately
from lemma \ref{lemma:part3}.

Now assume that $n\in \mathbb N$ is such that both statements hold
for all $m \in \mathbb{N}$ and $w \in \mathbb{B}(0,s_m)$. We fix
$m \in \mathbb N$ and $w \in \mathbb{B}\left(0,s_{m}\right)$.

We have $\bar{f}_m(w)\in \mathbb{B}\left(0,s_{m+1}\right)$, so by our induction hypothesis, we know that the lemma also holds for $n$, $m+1$ and $\bar{f}_m(w)$.
Since $\|\bar{f}_m(w)\| \leq s_{m+1}$, Lemma \ref{lemma:part3} shows that $\| \bar{h}_m \circ \bar{f}_m(w)\| \leq 2 s_{m+1} \leq r_{m+1}$. This puts us into position to apply Lemma \ref{lemma:part2}:
\begin{align*}
&\|\varphi_{m+(n+1),m} (w)
-\bar{g}_{m}^{-1}\circ
\bar{h}_{m+1}\left(\bar{f}_{m}(w) \right)\|\\
=&\|\bar{g}_{m}^{-1} \left(
\varphi_{(m+1)+n,m+1}
\left(\bar{f}_{m}(w)
\right)\right)
-\bar{g}_{m}^{-1}\left(
\bar{h}_{m+1}\left(\bar{f}_{m}(w) \right)\right)\|\\
\leq & \frac{1}{\lambda C} \left(\frac{M}{1-\frac{D^k}{\lambda
C}} \right) \|\bar{f}_{m}(w)\|^k\\
\leq & \frac{D^k}{\lambda C} \left(\frac{M}{1-\frac{D^k}{\lambda
C}} \right) \|w\|^k.
\end{align*}
Since $\|w\| \leq s_m \leq r_m$, we can also apply Lemma \ref{lemma:part4} (and the triangle inequality) to obtain
\begin{align*}
\|\varphi_{m+(n+1),m}(w)
-\bar{h}_{m}(w)\|
\leq & \frac{D^k}{\lambda C} \left(\frac{M}{1-\frac{D^k}{\lambda
C}} \right) \|w\|^k + M \cdot
\|w\|^k\\
\le & \left(\frac{M}{1-\frac{D^k}{\lambda C}}\right)\|w\|^k.
\end{align*}
Combining this with Lemma \ref{lemma:part3} we find that
\begin{align*}
\|\varphi_{m+(n+1),m}^{-1} (w)\|
\leq & \left(\frac{M}{1-\frac{D^k}{\lambda C}}+2\right)
\|w\|.
\end{align*}
\end{proof}

\begin{thm}\label{thm:convergentie}
The sequence $(\Phi_n)$ converges uniformly on each $V_u$.
\end{thm}
\begin{proof}
First, fix $\eta>0$. Since the function $\bar{g}_{v(u),1}^{-1}$ is uniformly
continuous on $\bar{\mathbb{B}}(0,1)$, we can find
$\theta$ such that for all $z,w \in
\mathbb{B}(0,1)$:
$$
\|z-w\| < \theta \Rightarrow
\|\bar{g}_{v(u),1}^{-1}(z)-\bar{g}_{v(u),1}^{-1}(w)\|<\frac{\eta}{2}
$$
Now pick $N \in \mathbb{N}$ such that
$$
\frac{M}{1-\frac{D^k}{\lambda C}}\left(\frac{D^k}{\lambda C}
\right)^{N-v(u)} \leq \theta
$$
Fix $z \in V_u$ and $n \geq N$. For any $m \geq v(u)$ we have $\bar{f}_{m,1}(z)\in \mathbb{B}(0,s_m)$. In this setting, Lemma \ref{lemma:afstanden}(i) shows that
\begin{align}
\|\varphi_{n,N} \left(\bar{f}_{N,1}(z)
\right)
-\bar{h}_{N}\left(\bar{f}_{N,1}(z) \right)\| \leq \left(\frac{M}{1-\frac{D^k}{\lambda C}} \right)
D^{k(N-u)}\label{eq:whatever}
\end{align}
And for $m=v(u), v(u)+1,\ldots , v(u)+N$, Lemma \ref{lemma:afstanden}(ii) gives us
\begin{align*}
\|\bar{g}_{N,m}^{-1}\circ
\varphi_{n,N} \left(\bar{f}_{N,1}(z) \right)\|
\leq  \left(\frac{M}{1-\frac{D^k}{\lambda C}}+2 \right)s_m=r_m,
\end{align*}
and
 \begin{align*}
\|\bar{g}_{N,m}^{-1}\circ
 \bar{h}_{N} \left(\bar{f}_{N,1}(z) \right)\|
\leq \left(\frac{M}{1-\frac{D^k}{\lambda C}}+2 \right)s_m=r_m.
 \end{align*}
If we apply $\bar{g}_{v(u)+N-1}^{-1}$ up to
$\bar{g}_{v(u)}^{-1}$ in equation (\ref{eq:whatever}) and use Lemma
\ref{lemma:part2} repeatedly we obtain:
\begin{align*}
\|\bar{g}_{N,v(u)}^{-1}\circ
\varphi_{n,N} \left(\bar{f}_{N,1}(z) \right)-\bar{g}_{N,v(u)}^{-1}\circ
 \bar{h}_{N} \left(\bar{f}_{N,1}(z) \right)\|
&\leq \left(\frac{D^k}{\lambda C}\right)^{N-v(u)}
\left(\frac{M}{1-\frac{D^k}{\lambda C}} \right)\\
& <\theta,
\end{align*}
which means that
$$
\|\varphi_{n,v(u)} \circ
\bar{f}_{v(u),1}(z) -\varphi_{N,v(u)} \circ \bar{f}_{v(u),1}(z) \|<\theta.
$$

Since both $\varphi_{n,v(u)} \circ \bar{f}_{v(u),1}(z)$ and
$\varphi_{N,v(u)} \circ \bar{f}_{v(u),1}(z)$ lie in
 $\mathbb{B}(0,1)$ by Lemma \ref{lemma:afstanden}(ii), the definition of $\theta$ gives:
 \begin{align*}
 \| \Phi_{n}(z)-\Phi_{N}(z)\|
& = \|\bar{g}_{v(u),1}^{-1}\left(\varphi_{n,v(u)} \circ \bar{f}_{v(u),1}(z)
\right)
 -\bar{g}_{v(u),1}^{-1}\left(\varphi_{N,v(u)} \circ \bar{f}_{v(u),1}(z)\right)\|\\
& \leq \eta/2.
\end{align*}

For $m,n \geq N$ and $z\in V_u$ we therefore have
$\|\Phi_{n}(z)-\Phi_{m}(z)\|<\eta$, which proves the uniform
convergence on $V_u$.
\end{proof}

Since $\Omega_{(\bar{f}_n)} = \bigcup_{u} V_u$, Theorem \ref{thm:convergentie} shows that the maps $\Phi_n$ converge uniformly on compact subsets to a map $\Phi \colon \Omega_{(\bar{f}_n)} \rightarrow \mathbb{C}^2$. We will now prove that this limit $\Phi$ maps $\Omega_{(\bar{f}_n)}$ biholomorphically onto $\mathbb C^2$. The maps $\Phi_n$ are compositions of a number of global biholomorphisms, plus a holomorphic map $\bar{h}_n$ which is injective on a small ball whose radius decreases with $n$. We first estimate the size of these radii.

\begin{lemma}\label{lemma:h}
For $n \in \mathbb N$ and $\|z\| \le r_n$ we have
\begin{enumerate}
\item[(i)]{$\|D\bar{h}_n-I\| \leq \frac{1}{2}$}
\item[(ii)]{$\bar{h}_n$ is injective, and}
\item[(iii)]{$\|\bar{h}_n(z)\| \geq \frac{1}{2}\|z\|$.}
\end{enumerate}
\end{lemma}
\begin{proof}
All three statements follow from the estimates on the coefficients of $h_n$ found earlier, and the fact that
$$
r_n \le \frac{1}{8Z}D^{2n(k-2+\epsilon)}.
$$
\end{proof}

\begin{corollary}
The map $\Phi$ is a biholomorphism.
\end{corollary}

\begin{proof}
Let $u \in \mathbb{N}$ and $n \ge v(u)$. Then for $z \in V_u$ we have $\|\bar{f}_{n,1}(z)\|\leq r_n$. By Lemma \ref{lemma:h}(ii) and the fact that all $\bar{f}_m$ and $\bar{g}_m$ are biholomorphisms, the map $\Phi_n$ must be injective on $V_u$. Therefore $\Phi|_{V_u}$ is a uniform limit of biholomorphisms, and we know that $D\Phi_n(0)=I$ for all $n$. By Hurwitz' theorem we can conclude that $\Phi|_{V_u}$ is also a biholomorphism. The statement follows since $\Omega_{(\bar{f}_n)}$ is the increasing union of the sets $V_u$.
\end{proof}

The final ingredient in the proof of Theorem \ref{thm:general} is to show that $\Phi: \Omega_{(\bar{f}_n)} \rightarrow \mathbb C^2$ is surjective.

\begin{lemma}\label{lemma:h2}
For $m , n \in \mathbb{N}$, we have
$$
\varphi_{m+n,m}\left( \mathbb{B}(0,s_m) \right) \supseteq \mathbb{B}\left( 0, \frac{1}{4}s_m \right).
$$
\end{lemma}
\begin{proof}
By Lemmas \ref{lemma:h}(iii) and \ref{lemma:afstanden}(i) it follows that
$$
\varphi_{m+n,m}(z)\ge \frac{1}{2}s_m - \left(\frac{M}{1-\frac{D^k}{\lambda C}} \right)
s_m^k \ge \frac{1}{4}s_m,
$$
for all  $z \in \partial \mathbb{B}(0,s_m)$. Now note that
$$
\bar{f}_{m+n,m}\left( \mathbb{B}(0,s_m) \right) \subseteq \mathbb{B}\left(0,s_m D^n \right) \subseteq \mathbb{B}\left(0,s_{m+n}\right).
$$
By Lemma \ref{lemma:h}(i), and the fact that all $\bar{f}_s$ and $\bar{g}_s$ are biholomorphisms, the map $\varphi_{m+n,m}$ has a nonzero Jacobian on $\mathbb{B}(0,s_m)$, and therefore is an open mapping. As $\varphi_{m+n,m}(0)=0$, the lemma follows.
\end{proof}

\begin{lemma}
The biholomorphism $\Phi: \Omega_{(\bar{f}_n)} \rightarrow \mathbb C^2$ is surjective.
\end{lemma}
\begin{proof}
Let $w \in \mathbb{C}^2$. We will show that $w$ lies in the image of $\Phi$.

Since $\Omega_{(\bar{g}_n)}=\mathbb{C}^2$, we can find $t \in \mathbb{N}$ such that $\bar{g}_{t,1}(w) \in \mathbb{B}(0,\frac{1}{4})$. Then we have
$$
\|\bar{g}_{v(t),1}(w)\|=\|\bar{g}_{v(t),t}\left(\bar{g}_{t,1}\right)\| \leq \frac{1}{4}D^{v(t)-t} \leq \frac{1}{4}s_{v(t)}.
$$
By Lemma \ref{lemma:h2} we now find that
$$
\bar{g}_{v(t),1}(w) \in \varphi_{v(t)+n,v(t)}\left( \mathbb{B}(0,s_v(t)) \right),
$$
for any $n \in \mathbb{N}$. This implies that
\begin{align*}
w & \in \bar{g}_{v(t),1}^{-1} \circ \varphi_{v(t)+n,v(t)}\left( \mathbb{B}(0,s_v(t)) \right)=\Phi_{v(t)+n}\left( \bar{f}_{v(t),1}^{-1}\left( \mathbb{B}(0,s_v(t))\right) \right)\\
& \subseteq \Phi_{v(t)+n}\left( \bar{f}_{v(t),1}^{-1}\left( \bar{\mathbb{B}}(0,s_v(t))\right) \right),
\end{align*}
for any $n \in \mathbb{N}$. By the compactness of $\bar{f}_{v(t),1}^{-1}\left( \bar{\mathbb{B}}(0,s_v(t))\right)$ and the uniform convergence of $(\Phi_n)$ on compact sets we have that
$$
w \in \Phi\left(\bar{f}_{v(t),1}^{-1}\left( \bar{\mathbb{B}}(0,s_v(t))\right)\right).
$$
\end{proof}

With this lemma we have proved Theorem \ref{thm:general}.


\begin{thebibliography}{9999}
\bibitem{AAM} Abate, M, Abbondandolo, A, Majer, P., \emph{Stable manifolds for holomorphic automorphisms}, J. Reine Angew. Math. (to appear), published online (2012).

\bibitem{AM} Abbondandolo, A, Majer, P., \emph{Global stable manifolds in holomorphic dynamics under bunching conditions}, Int Math Res Notices (to appear), published online (2013).

\bibitem{Arosio} Arosio, L., \emph{Basins of attraction in {L}oewner equations}, Ann. Acad. Sci. Fenn. Math. {\bf 37}, (2012), no. 2, 563--570.

\bibitem{ABW} Arosio, L., Bracci, F., Wold, E. F., \emph{Solving the Loewner PDE in complete hyperbolic starlike domains of $\mathbb{C}^N$.} Adv. Math. {\bf 242} (2013), 209--216.

\bibitem{Bedford} Bedford, E., \emph{Open problem session of the Biholomorphic Mappings Meeting at the American Institute of Mathematics}, Palo Alto, CA, July, (2000).

\bibitem{Berteloot} Berteloot, F., \emph{M\'ethodes de changement d'\'echelles en analyse complexe}. Ann. Fac. Sci. Toulouse Math. (6) {\bf 15}, (2006), no. 3, 427--483.

\bibitem{BDM} Berteloot, F., Dupont, C., Molino, L., \emph{Normalization of random families of holomorphic contractions and applications to dynamics}, Ann. Inst. Fourier (Grenoble) {\bf 58} (2008), no. 6, 2137--2168.

\bibitem{Fornaess} Forn{\ae}ss, J.E., \emph{Short $\mathbb C^k$}, Advanced Studies in Pure Mathematics, Mathematical Society of Japan, (2003).

\bibitem{FS} Forn{\ae}ss, J.E., Stens{\o}nes, B., \emph{Stable manifolds of holomorphic hyperbolic maps}, Internat. J. Math. {\bf 15} (2004), no. 8, 749--758.

\bibitem{JV} Jonsson, M., Varolin, D., \emph{Stable manifolds of holomorphic diffeomorphisms}, Invent. Math. 149, (2002) 409--430.

\bibitem{Peters} Peters, H., \emph{Perturbed basins}  Math. Ann. {\bf 337} (2007), no. 1, 1--13.

\bibitem{PW} Peters, H., Wold, E. F., \emph{Non-autonomous basins of attraction and their boundaries}, {J. Geom. Anal.} {\bf 15}, (2005), no. 1, 123--136.

\bibitem{RR} Rosay, J.P., Rudin, W., \emph{Holomorphic maps from $\mathbb C^n$ to $\mathbb C^n$}, Trans. Amer. Math. Soc. {\bf 310} (1988), no. 1, 47--86.

\bibitem{Sabiini} Sabiini, G., \emph{Les suites de contractions holomorphes et domaines de Fatou-Bieberbach}, PhD Thesis, Universit\'e Toulouse III - Paul Sabatier, 2010, available     online at \url{http://thesesups.ups-tlse.fr/804/1/Sabiini_Guitta.pdf}.

\bibitem{Stensones} Stens{\o}nes, B., \emph{Fatou-Bieberbach domains with $C^\infty$-smooth boundary}, Ann. of Math. {\bf 145}, (1997), 365--377.

\bibitem{Sternberg} Sternberg, S., \emph{Local contractions and a theorem of Poincar\'e}, Amer. J. Math. {\bf 79}, (1957), 809--824.

\bibitem{Wold} Wold, E. F., \emph{Fatou-Bieberbach domains,} Internat. J. Math. {\bf 16} (2005), no. 10, 1119--1130.

\bibitem{Wold2} Wold, E. F., \emph{A Fatou-Bieberbach domain in $\mathbb C^2$ which is not Runge}, Math. Ann. {\bf 340} (2008), no. 4, 775--780.

\bibitem{Wolf} Wolf, C., \emph{Dimension of Julia sets of polynomial automorphisms of $\mathbb C^2$}, Mich. Math. J. {\bf 47}, (2000), 585--600.

\end{thebibliography}
\end{document}